\documentclass[11pt,oneside,english,reqno]{amsart}
\usepackage[T1]{fontenc}
\usepackage[utf8]{inputenc}
\usepackage[a4paper]{geometry}
\usepackage{mathrsfs}
\usepackage{amstext}
\usepackage{amsthm}
\usepackage{amssymb}
\usepackage{color}
\usepackage{hyperref}
\usepackage{showlabels,dsfont}
\usepackage{enumerate}
\usepackage{verbatim} 
\usepackage{stmaryrd}
\usepackage{enumitem}
\usepackage{graphicx}

\makeatletter
\numberwithin{equation}{section}
\numberwithin{figure}{section}
\theoremstyle{plain}
\newtheorem{thm}{\protect\theoremname}
\theoremstyle{definition}
\newtheorem{defn}[thm]{\protect\definitionname}
\theoremstyle{remark}
\newtheorem{rem}[thm]{\protect\remarkname}
\theoremstyle{plain}
\newtheorem{lem}[thm]{\protect\lemmaname}
\theoremstyle{plain}
\newtheorem{prop}[thm]{\protect\propositionname}
\theoremstyle{plain}
\newtheorem{cor}[thm]{\protect\corollaryname}

\makeatother

\usepackage{babel}
\providecommand{\corollaryname}{Corollary}
\providecommand{\definitionname}{Definition}
\providecommand{\lemmaname}{Lemma}
\providecommand{\propositionname}{Proposition}
\providecommand{\remarkname}{Remark}
\providecommand{\theoremname}{Theorem}

\newcommand{\cB}{\mathcal{B}}
\newcommand{\cC}{\mathcal{C}}

\newcommand{\cF}{\mathcal{F}}

\newcommand{\cH}{\mathcal{H}}
\newcommand{\cI}{\mathcal{I}}

\newcommand{\cM}{\mathcal{M}}

\newcommand{\cP}{\mathcal{P}}

\newcommand{\cS}{\mathcal{S}}

\newcommand{\NN}{\mathbb{N}}

\newcommand{\PP}{\mathbb{P}}

\newcommand{\RR}{\mathbb{R}}

\newcommand{\dd}{\mathop{}\!\mathrm{d}}

\newcommand{\norm}[1]{\left\lVert #1\right\rVert}

\begin{document}
	
	\title[regularization of the wave equation]{Non-linear Young equations in the plane and pathwise regularization by noise for  the stochastic wave equation }
	\author{Florian Bechtold \and Fabian A. Harang \and Nimit Rana}
	
	\address{Florian Bechtold, Nimit Rana: Fakult\"at f\"ur Mathematik,
		Universit\"at Bielefeld, 33501 Bielefeld, Germany}
	\email{fbechtold@math.uni-bielefeld.de, nrana@math.uni-bielefeld.de}
	
	\address{Fabian A. Harang: Department of Economics, BI Norwegian Business School, Handelshøyskolen BI, 0442, Oslo, Norway.}
	\email{fabian.a.harang@bi.no}


	
	\begin{abstract}
		We study pathwise regularization by noise for equations on the plane in the spirit of the framework outlined by Catellier and Gubinelli in  \cite{Catellier2016}. To this end, we extend the notion  of non-linear Young equations to a two dimensional domain and prove existence and uniqueness of such equations. This concept is then used in order to prove regularization by noise for stochastic equations on the plane. The statement of regularization by noise is  formulated in terms of the regularity of the local time associated to the perturbing stochastic field. For this, we provide two quantified example: a fractional Brownian sheet and the sum of two one-parameter fractional Brownian motions.
		As a further illustration of our regularization results, we also prove well-posedness of a  1D non-linear wave equation with a noisy boundary given by fractional Brownian motions.
		A discussion of open problems and further investigations is provided.
	\end{abstract}

	\maketitle
	
	\tableofcontents{}
	
	\section{Introduction}
	
	\emph{Regularization by noise} is the study of  the potentially regularizing effect of irregular paths or stochastic processes on a-priori ill-posed Ordinary Differential Equations (ODEs) or Partial Differential Equations (PDEs). For illustration, consider the following differential equation
	\begin{equation}\label{eq:intro}
		\dd x_t = b(x_t)\dd t +\dd z_t, \quad t \geq 0,
	\end{equation}
	where $b$ is a non-linear function and $z$ is a continuous path. Zvonkin \cite{Zvonkin1974} initially observed that when $z$ is a Brownian motion, and the above equation is interpreted as a Stochastic Differential Equation (SDE) then strong existence and uniqueness holds even when $b$ is merely bounded and measurable. This is in contrast to the classical theory of ODEs (when $z\equiv 0$ in \eqref{eq:intro}), where one typically requires that $b$ is Lipschitz (or of similar regularity) in order to guarantee uniqueness, see e.g. \cite{Coddington1961}.  Similar regularization by noise phenomena were proved by Davie in  \cite{Davie2007_Unique}  where, in contrast to \cite{Zvonkin1974},   uniqueness of solutions to \eqref{eq:intro} was proven in a path-by-path manner when $b$ is a bounded and measurable function and $z$ is a continuous path sampled from the law of the Brownian motion. 
	We can therefore think of  $z$ as a process which might provide a regularizing effect on the drift coefficient $b$ such that a unique solution can be proven to exist, even when existence and/or uniqueness  fails in the classical setting ($z\equiv 0$). This field of study has since the initiation in \cite{Zvonkin1974} seen a rapid development, and investigations into the regularizing effect of various stochastic processes are by now a large field of research. 
	
	There is typically a clear distinction in the approach to proving such effects. The traditional approach is based on classical tools from stochastic analysis and probability theory, and the concept of solutions which are then investigated is in the sense of  (probabilistic) strong or weak solutions, see e.g. \cite{FGP2010}. More recently, by using  tools inspired by the theory of rough paths, much progress has been made towards understanding the path-by-path or pathwise regularization by noise effect. 
	The idea is first to identify a class of stochastic processes whose paths provide the desired regularization effect, and then solve the SDE \eqref{eq:intro} in a pathwise manner, see e.g. \cite{Catellier2016,galeati2020noiseless,HarangPerkowski2020} for some of the recent works in this direction. With the approach presented in these papers, the authors are able to give meaning to and prove pathwise  wellposedness of equations of the form of \eqref{eq:intro} even in the case when $b$ is truly a distribution (in the sense of generalized functions). A particularly interesting feature of this approach is the apparent connection between the regularization by noise effect and the regularity of the local time associated to noise source in the equation. 
	
	Regularization by noise has also been investigated extensively in the context of SPDEs. While there are certainly several papers studying the regularization by noise effect for parabolic SPDEs both from a probabilistic and pathwise perspective, see e.g. \cite{mueller,neuman,nualartyoussef}, and more recently \cite{Athreya20,joern,CatellierHarang21}, we will discuss here some of the development for regularization by noise for hyperbolic SPDEs, more closely related to the equations considered  in this article.

	Motivated by the pathwise techniques used for regularization by noise in \cite{Catellier2016,galeati2020noiseless,HarangPerkowski2020},  we will in this article extend the aforementioned techniques in order to prove pathwise regularization by noise results for stochastic differential equations on the plane of the form 
	\begin{equation}\label{eq:integral eq intro}
		x_t =\xi_t +\int_0^t b(x_s)\dd s + w_t, \quad t=(t_1,t_2)\in [0,T]^2,
	\end{equation}
	where $w:[0,T]^2\rightarrow \RR^d$ is an additive continuous field, and we use the notation $\int_0^t b(x_s)\dd s:=\int_0^{t_1}\int_0^{t_2} b(x_{s_1,s_2})\dd s_1\dd s_2$. The term $\xi:[0,T]^2\rightarrow \RR^d$ is a continuous field representing the boundary conditions of the equation. 
	In contrast to the probabilistic methodology used for regularization by noise for the stochastic heat equation in \cite{Athreya20} based on stochastic sewing lemma, our analysis of regularization will be done through an analysis of the regularity of the local time associated to the additive continuous field $w$, similarly as done for ODEs in \cite{Harang2021RegularityOL,HarangPerkowski2020}. In combination with an extension of the theory of non-linear Young equations to the plane, this will allow for a purely pathwise analysis of the regularization by noise phenomena for a class of hyperbolic equations. When considering stochastic equations,  this methodology can be interpreted in the spirit of rough paths theory: probabilistic considerations are only needed in order to prove pathwise regularity of the local time, and the analysis of the equation itself is done purely analytically.  While proving space time regularity of various stochastic processes has recently  been extensively investigated in several articles  \cite{galeati2020prevalence,Harang2021RegularityOL,HarangPerkowski2020}, similar analysis for stochastic fields have not yet received equal recent attention. In particular, with respect to the our framework, only partial results in this direction are known (see e.g. \cite{GemanHorowitz1980} and the discussion in Remark \ref{discussion local time} explaining why such results are not sufficient in our context). In this article we will focus on the analytical step of the pathwise regularization by noise program described above, while also providing new probabilistic space-time regularity estimates for the fractional Brownian sheet that allow to employ the analytic machinery developed (refer to Theorem \ref{regularity fBS}. Let us already mention at this point however, that we do not expect this probabilistic result to be optimal and further research in this direction seems to be in place (refer also the the final Section \ref{sec:challenges} for a brief discussion of potential approaches in this context). 
	
	Before motivating our approach in more detail, observe that the integral equation can be seen to be the integrated version of the so-called Goursat partial differential equation 
	\begin{equation*}
		\frac{\partial^2}{\partial t_1\partial t_2 } x_{t} =b(x_t)+\frac{\partial^2}{\partial t_1\partial t_2 }w_t, \quad t=(t_1,t_2)\in [0,T]^2,
	\end{equation*}
	with the boundary conditions $x_{(0,t_2)}=\xi_{(0,t_2)}$, $x_{(t_1,0)}=\xi_{(t_1,0)}$ and $\frac{\partial^2}{\partial t_1\partial t_2 } \xi=0$, and $w$ is zero on the boundary of $[0,T]^2$ (i.e. $w_{0,t_2}=w_{t_1,0}=0$ for all $t\in [0,T]^2$). This hyperbolic equation is fundamentally linked with the stochastic wave equation, which will be illustrated in detail below, see Theorem \ref{thm-wave equation intro}. 
	Furthermore, we see the regularization by noise problem for the  integral equation in \eqref{eq:integral eq intro} as a first step in order to prove regularization by noise for more complicated SPDEs driven by a stochastic field,
	see in particular Section \ref{sec:challenges} for a discussion of further development and open problems. 
	
	The integral equation in  \eqref{eq:integral eq intro} has been extensively studied from a probabilistic point of view, mostly in the setting where $w$ is a Brownian sheet, but also other processes have been considered. Specifically, in \cite{Yeh81,Yeh87} strong existence and pathwise uniqueness of solutions to equations of the form 
	\begin{equation*}
		x_t=\xi_t + \int_0^t b(x_s)\dd s + \int_0^t \sigma(x_s)\dd w_s, \quad t\in [0,T]^2, 
	\end{equation*}
	under the assumption that both $b$ and $\sigma$ are Lipschitz continuous and of linear growth, and $w$ is a Brownian sheet.  The same author proved existence of weak solutions to the above equation when $b$ is merely continuous together with a certain growth condition and a condition on the sixth moment of the boundary process $\xi$.  
	Strong existence and uniqueness has later been obtained in \cite{Nualart2003} when the drift $b$ is bounded and nondecreasing, and  $w$ is a rough fractional Brownian sheet ( i.e. with Hurst parameters $H_1,H_2\leq \frac{1}{2}$, see section \ref{sec:regularization SDE on plane} for further information about the fractional Brownian sheet).
	
	Based on a multi-parameter version of the sewing lemma constructed in \cite{Harang2020}, we will in this article  extend the framework of non-linear Young equations used in \cite{Catellier2016,galeati2020noiseless,HarangPerkowski2020} (see also \cite{Galeati2020NonlinearYD} for a complete overview in the one-parameter case) to two-parameter processes. 
	Given a function $A:[0,T]^2\times \RR^d \rightarrow \RR^n$ which is sufficiently regular in both time  and spatial arguments, and a sufficiently regular path $y:[0,T]^2\rightarrow \RR^d$, one can then construct a nonlinear Young integral on the plane of the form 
	\begin{equation*}
		\int_s^t A(\dd r,y_r)=\lim_{|\cP|\rightarrow 0} \sum_{ [u,v]\in \cP} \square_{u,v} A(\cdot, y_u),
	\end{equation*}
	where $\cP$ is a partition of  the rectangle $[s_1,t_1]\times[s_2,t_2]$ consisting of rectangles of the form $[u,v]=[u_1,v_1]\times [u_2,v_2]$ for $[u_i,v_i]\subset [s_i,t_i]$ with $i=1,2$,  and the limit is taken as the mesh of the partition goes to zero. The operator $\square$ denotes the rectangular increment and is defined for $s,t\in [0,T]^2$ by 
	\begin{equation*}
		\square_{s,t} f : =f(t_1,t_2)-f(t_1,s_2)-f(s_1,t_2)+f(s_1,s_2), 
	\end{equation*}
	for any function $f$ on the plane $[0,T]^2$. 
	With the goal of proving pathwise existence and uniqueness of \eqref{eq:integral eq intro} when $b$ is a distribution, a crucial step will be to give meaning to the integral term.
	By setting $\theta=x-w$ we observe that, formally, $\theta$ solves the equation  
	\begin{equation}\label{eq:transformed intro}
		\theta_t=\xi_t+\int_0^t b(\theta_s+w_s)\dd s. 
	\end{equation}
	The integral appearing on the right hand side may be easier to handle due to its connection with the local time associated to the path $w$. Indeed, recall that the local time formula tells us that for each $x\in \RR^d$ we have
	\begin{equation*}
		\int_s^t b(x+w_r)\dd r = (b\ast \square_{s,t} L^{-w})(x),
	\end{equation*}
	where $\ast$ denotes the usual convolution. We then observe that by Young's convolution inequality in Besov spaces, the regularity of the mapping  $x\mapsto \int_s^t b(x+w_r)\dd r $ is given as a sum of the spatial regularity exponents of $b$ and $L^{-w}$. Thus, if $L^{-w}$ is a sufficiently smooth function, the convolution $b\ast L^{-w}$ may be a differentiable function, even when $b$ is a distribution.  
	
	The nonlinear Young integral can therefore be used to give meaning to the integral appearing in \eqref{eq:transformed intro} by setting $A(t,x)=b\ast L^{-w}_t(x)$. If $\theta:[0,T]^2\rightarrow \RR^d$ is a sufficiently regular path, we  then define  
	\begin{equation}\label{eq:NLYintegral intro}
		\int_0^t b(\theta_s+w_s)\dd s=\int_0^t (b\ast L^{-w}) (\dd s,\theta_s)=\lim_{|\cP|\rightarrow 0} \sum_{[u,v]\in \cP} b\ast \square_{u,v} L^{-w}(\theta_u).
	\end{equation}
	In subsequent sections we show also that this integral coincides with the classical Riemann integral whenever $b$ is a continuous function. 
	Now existence and uniqueness of \eqref{eq:transformed intro} is granted under sufficient regularity conditions on $b$ and the local time $L^{-w}$.  We illustrate this by highlighting one of the main results in this article, after some notational definitions.

	Let $E$ be a Banach space. For $\gamma\in (0,1)^2$  we denote by $C^\gamma([0,T]^2; E)$ the set of $E$-valued two parameter {\em jointly} $\gamma$-H\"older continuous functions, proposed in Definition \ref{eq: 2D Holder space}. We denote the usual Besov spaces by $B_{p,q}^\alpha(\RR^d;\RR^n)$, refer Section \ref{sec:notation} for their definition. The first main result about of the current paper can be stated as follows, see Theorem \ref{thm:existence and uniqueness} for the precise
	formulation.
	\begin{thm}\label{thm-existence and uniqueness intro}
		Let $T>0$. Consider parameters $\alpha,\zeta\in \RR$, $\gamma=(\gamma_1,\gamma_2)\in (\frac{1}{2},1)^2$ and $\eta=(\eta_1,\eta_2)\in (0,1]^2$ such that for $i=1,2$ the following two conditions hold
		\begin{equation*}
			\alpha+\zeta\geq 2+\eta_i,\quad \mathrm{and} \quad (1+\eta_i)\gamma_i>1.  
		\end{equation*}
		Let $p,q\in [1,\infty]$ with $\frac{1}{p}+\frac{1}{q}=1$, and  suppose  that $w\in C([0,T]^2;\RR^d)$ has an associated local time $L^{-w}\in C^\gamma([0,T]^2; B^\alpha_{q,q}(\RR^d;\RR^d))$.  Then for every $\xi\in C^\gamma([0,T]^2;\RR^d)$, and any $b\in B^\zeta_{p,p}(\RR^d;\RR^d)$  there exists a unique solution $x\in C([0,T]^2;\RR^d)$ to equation \eqref{eq:integral eq intro}. More precisely, there exists a $\theta\in C^\gamma([0,T]^2;\RR^d)$ with the property that $x=w+\theta$, and $\theta$  satisfies 
		\begin{equation*}\label{eq:thm1 NLY regularized}
			\theta_t=\xi_t+\int_0^t (b\ast L^{-w})(\dd s,\theta_s), \quad t\in [0,T]^2,
		\end{equation*}
		where the integral is interpreted in the nonlinear Young sense, as in \eqref{eq:NLYintegral intro}.
	\end{thm}
	\begin{rem}
		Note that Uniqueness then holds in the class of functions of the form $x=w+\theta$ with $\theta\in C^\gamma$ satisfying \eqref{eq:thm1 NLY regularized}.
	\end{rem}
	As is clear from the above theorem, a second crucial ingredient in our study of regularization by noise in the plane consists of regularity estimates for the local time time associated with stochastic fields. Our second main result can be stated as follows (see Theorem \ref{regularity fBS})
	\begin{thm}
		\label{local time intro}
		Let $w:[0,T]^2\to \RR^d$ be a fractional Brownian sheet of Hurst parameter $H=(H_1, H_2)$ on $(\Omega, \mathcal{F}, \mathbb{P})$. Suppose that 
		\[
		\lambda<\frac{1}{2(H_1\vee H_2)}-\frac{d}{2}.
		\]
		Then for almost every $\omega\in \Omega$, $w$ admits a local time $L$ such that for $\gamma_1\in (1/2, 1-(\lambda+\frac{d}{2})H_1)$ and $\gamma_2\in (1/2, 1-(\lambda+\frac{d}{2})H_2)$ 
		\[
		\norm{\square_{s,t} L}_{H^\lambda_x}\lesssim (t_1-s_1)^{\gamma_1}(t_2-s_2)^{\gamma_2}
		\]
	\end{thm}
	Combining Theorems \ref{thm-existence and uniqueness intro} and \ref{local time intro}, we immediately obtain a regularization by noise result for stochastic differential equations in the plane perturbed by an additive fractional Brownian sheet. 
	As a further interesting application of  Theorem \ref{thm-existence and uniqueness intro}, we also study the Goursat boundary regularization of wave equation with singular non-linearities.  More precisely, for a distributional non-linearity $h$, we study the problem
	\begin{equation}\label{eq: wave eq, w/noisy boundary intro}
		\begin{split}
			\left( \frac{\partial^2}{\partial x^2}- \frac{\partial^2}{\partial y^2}\right)u=h(u(x,y)),
		\end{split}
	\end{equation}
	on $(x, y)\in R_{\frac{\pi}{4}}\circ [0,T]^2$  subject to the Goursat boundary conditions
	\begin{equation}
		\begin{split}
			u(x,y)&=\beta^1(y) \qquad \text{if}\ y=x\\
			u(x,y)&=\beta^2(y) \qquad \text{if}\ y=-x.
		\end{split}
		\label{boundary wave eqn intro}
	\end{equation}
	Here $R_{\frac{\pi}{4}}$ denotes the rotation operator of the plane by $\pi/4$. 
	
	The wave equation in \eqref{eq: wave eq, w/noisy boundary intro} with random boundary conditions \eqref{boundary wave eqn intro} arises in the literature on the splitting method to construct a non-continuous approximation of the solution of a stochastic Goursat problem, see e.g. \cite{AGW1999}. Thus, in our understanding, the analysis of this problem will help to apply the splitting method to study Goursat problem with a distributional non-linearity perturbed by a suitable fractional Brownian sheet. However, in case $h$ is a true distribution it is even unclear how to define a solution to this problem and prove its well-posedness by using classical methods. 
	To circumvent this issue, the main observation in our analysis is that the one-dimensional wave equation \eqref{eq: wave eq, w/noisy boundary intro} with the above prescribed Goursat boundary condition \eqref{boundary wave eqn intro} can be transformed into a Goursat PDE, accessible by our previous Theorem \ref{thm-existence and uniqueness intro}, as shown in Section \ref{sec:wave eqn}. 
	In particular, assuming $\beta^1, \beta^2$ to be sufficiently regularizing meaning they admit sufficiently regular local times, said Theorem \ref{thm-existence and uniqueness intro} allows to establish existence and uniqueness to the problem 
	\begin{equation}
		\label{limit ode intro}
		\psi_t=-2\int_0^t (h*L)(ds, \psi_s),
	\end{equation}
	where $L_t(x)=L^{-\beta^1(\cdot/\sqrt{2})}*L^{-\beta^2(\cdot/\sqrt{2})}(x)$, and the integral is interpreted in the non-linear Young sense. Then the second main result 
	result of this paper is as follows,  see Theorem \ref{thm-wave equation} for details and the precise formulation.
	\begin{thm}\label{thm-wave equation intro}
		Assume that the parameters $p,q,\zeta,\gamma$ and $\alpha$ belong to the range as in Theorem \ref{thm-existence and uniqueness intro}. Suppose $h\in B^\zeta_{p,p}(\RR^d)$ and suppose that $L_t(x)=(L^{-\beta^1(\cdot/\sqrt{2})}_{t_1}*L^{-\beta^2(\cdot/\sqrt{2})}_{t_2})(x)$ satisfies $L\in C^\gamma_tB^\alpha_{q,q}(\RR^d)$. 
		Then there exists a unique solution $u$ to \eqref{eq: wave eq, w/noisy boundary intro}, in the sense of Definition \ref{notion of solution wave}, given by %
		\[
		u(x,y):=\psi(\frac{y+x}{\sqrt{2}}, \frac{y-x}{\sqrt{2}})+\beta^1(\frac{y+x}{\sqrt{2}})+\beta^2(\frac{y-x}{\sqrt{2}}),
		\]
		where $\psi$ is the unique solution to \eqref{limit ode intro}, interpreted in the nonlinear Young sense. 
	\end{thm}
	\begin{rem}
		In Theorem \ref{thm:fbm reg} we provide specific conditions under which sample paths of the fractional Brownian motion may be used as boundary processes $\beta^1$ and $\beta^2$ thus providing one example of  stochastic paths which fulfills the conditions in Theorem \ref{thm-wave equation intro}. However, the class of stochastic processes providing such regularizing effects is by now well studied, see e.g. \cite{galeati2020prevalence,Harang2021RegularityOL,HarangPerkowski2020}, and we therefore do not study such processes in more detail here. 
	\end{rem}
	\begin{rem}
		Note that by developing a Young integration theory in two dimensions, whose proof is on similar lines, as presented in Proposition \ref{prop. 2d NLY integral}, the authors of \cite{Quer-Tindel2007} have shown the existence of a unique solution to a non-linear one dimensional wave equation driven by an arbitrary signal whose rectangular increments satisfy some H\"older  continuous with H\"older exponent greater than $\frac{1}{2}$. Similar results for one dimensional stochastic geometric wave equation have also been presented by the last author of this paper, in collaboration with Brze\'zniak,  in  \cite{Nimit2020} where the noise is modelled by a fractional Brownian sheet of Hurst parameters greater than $\frac{3}{4}$. In contrast to Proposition \ref{prop. 2d NLY integral}, to achieve the existence of a unique local solution they extend the theory of pathwise stochastic integrals in Besov spaces to two dimensional setting.  However, the present work focuses on the regularization by boundary conditions in the context of 1D wave equation (which in integral form can be seen as an additive perturbation of the wave equation), our Theorem~\ref{thm-wave equation intro} is fundamentally different from the main results of \cite{Nimit2020,Quer-Tindel2007} where a multiplicative noise is considered. Hence, the results presented in this article are not comparable with their results in a straightforward manner.
	\end{rem}

	This paper is organized in six sections. Section \ref{sec:notation} covers the notation and the required definitions used in the paper. 
	Section \ref{sec:2DNLY} is devoted to the extension of nonlinear Young theory to two dimensional integrands, and in particular prove existence, uniqueness and stability of  non-linear Young integral equations. In Section \ref{sec:regularization SDE on plane} we give a rigorous concept of solution to \eqref{eq:integral eq intro} and prove the regularization by noise effect under certain conditions on the local time associated to the noise.  We moreover provide a quantitative regularity estimate for the local time associated with the fractional Brownian sheet which can then be employed in the study of the aforementioned regularization by noise phenomenon. 
	In Section \ref{sec:wave eqn} we demonstrate how the theory of 2D nonlinear Young equations can be employed in the study of regularization of the wave equation with a noisy Goursat  type boundary condition. In particular, wellposedness of this equation when the nonlinear coefficient is a distribution and the boundary processes are given as rough fractional Brownian motions is proven. 
	We conclude the paper with Section \ref{sec:challenges} in which we discuss further extensions of our results and other related challenging open problems.

	\section{Notation} \label{sec:notation}
	We will work with a partial ordering of points in the rectangle $[0,T]^2$, in the sense that for $s,t\in [0,T]^2$ the notation  $s<t$ means that $s_1<t_1$ and $s_2<t_2$.
	We will work in a two-parameter setting, and will therefore frequently work with rectangles as opposed to intervals. For $s=(s_1, s_2)$ and $t=(t_1, t_2)$ with $s<t$, we define  $[s,t]\subset [0,T]^2$  by $[s,t]=[s_1,t_1]\times[s_2,t_2]$. We therefore consider $[s,t]$ to be the rectangle spanned by the lower left point $(s_1, s_2)$ and the upper right point $(t_1, t_2)$.   We will refer to the set $\{0\}\times [0,T]$ and $[0,T]\times \{0\}$ as the boundary of $[0,T]^2$. 
	For two numbers $a$ and $b$ the notation  $a\lesssim b$ (or $a\sim b$) means that there exists a constant $C>0$ such that $a\leq Cb$ (or $(a=Cb$). If the constant $C$ depends on an important parameter $k$ we use the notation $\lesssim_k$ (or $\sim_k$). 
	
	For a function $A:[0,T]^2\times [0,T]^2\rightarrow \RR^d$, we set
	\[
	\square_{s,t}A:=A_{(s_1, s_2), (t_1, t_2)}.
	\]
	We will denote the increment of a function $f:[0,T]^2\rightarrow \RR^d $ over a rectangle $[s,t]$ by 
	\begin{equation*}
		\square_{s,t} f=f_{t_1,t_2}-f_{t_1,s_2}-f_{s_1,t_2}+f_{s_1,s_2},
	\end{equation*}
	which canonically generalizes the notion of an increment in the two dimensional setting.
	This type of increment satisfies certain important properties which will be used throughout the article, and we therefore comment on some of these properties here. 
	
	If the mixed partial derivative $\frac{\partial^2 f(t_1,t_2)}{\partial t_1\partial t_2} $ exists for all $t\in [0,T]^2$, then it is readily seen that 
	\begin{equation*}
		\square_{s,t} f = \int_s^t \frac{\partial^2 f(r_1,r_2)}{\partial r_1\partial r_2}\dd r, 
	\end{equation*}
	where we use the double-integral notation $\int_s^t :=\int_{s_1}^{t_1}\int_{s_2}^{t_2}$ and $\dd r=\dd r_2 \dd r_1$. Furthermore, we observe that if $g(t_1,t_2)=\square_{0,t}f$, then $g$ is zero on the boundary, since 
	$$ g(t_1,0)=\square_{0,(t_1,0)}f=f(t_1,0)-f(0,0)-f(t_1,0)+f(0,0)=0,
	$$
	and similarly we can check that $g(0,t_2)=0$. Furthermore, 
	we have 
	\begin{equation}\label{eq:increment f and g}
		\square_{s,t} g=\square_{s,t} f. 
	\end{equation}
	Note that the two functions can still be different on the boundary, as this is not captured by the rectangular increment.

	We will work with a 2D H\"older space, capturing the necessary regularity of fields of interest in each of their variables. To this end, we also introduce two concepts which will be used to measure the regularity: Namely, for $s<t\in [0,T]^2$ and $\alpha=(\alpha_1,\alpha_2)\in (0,1)^2$ we define 
	\begin{equation}\label{eq:dist multi}
		m(t-s)^\alpha := |t_1-s_1|^{\alpha_1}|t_2-s_2|^{\alpha_2}. 
	\end{equation}
	With a slight abuse of notation we also define 
	\begin{equation}\label{eq:dist}
		|t-s|^\alpha=|t_1-s_1|^{\alpha_1}+|t_2-s_2|^{\alpha_2}. 
	\end{equation}
	
	\begin{defn}\label{eq: 2D Holder space}
		Let $E$ be a Banach space and $f:[0,T]^2\rightarrow E$ be such that, for some $\alpha = (\alpha_1,\alpha_2)\in (0,1)^2$,
		\begin{equation*}
			[f]_\alpha:=[f]_{(1,0),\alpha}+[f]_{(0,1),\alpha}+[f]_{(1,1),\alpha} <\infty, 
		\end{equation*}
		where we define the semi-norms 
		\begin{equation}\label{eq:holder semi 2d}
			\begin{aligned}
				\,[f]_{(1,0),\alpha}&:=\sup_{s\neq t\in [0,T]^2} \frac{|f(t_1,s_2)-f(s_1,s_2)|_E}{|t_1-s_1|^{\alpha_1}},
				\\
				[f]_{(0,1),\alpha}&:=\sup_{s\neq t\in [0,T]^2} \frac{|f(s_1,t_2)-f(s_1,s_2)|_E}{|t_2-s_2|^{\alpha_2}},
				\\
				[f]_{(1,1),\alpha}&:=\sup_{s\neq t\in [0,T]^2} \frac{|\square_{s,t} f|_E}{m(t-s)^\alpha}.
			\end{aligned}
		\end{equation}
		We then say that $f$ is $\alpha$-H\"older continuous on the rectangle $[0,T]^2$, and we write $f\in C^\alpha_t E$.
		Under the mapping $f\mapsto |f(0,0)|+[f]_\alpha =: \| f\|_{C^\alpha_t E}$, the space $C^\alpha_t E$ is a Banach space. Whenever $E=\RR^d$ we write $C^\alpha_t$ or sometimes $C^\alpha([0,T]^2;\RR^d)$ instead of $C^\alpha_t \RR^d$. Moreover, if we need to keep track of the interval over which compute the above quantities then we will highlight the interval explicitly in subscript, for e.g. $[f]_{\alpha,[0,T]}$.
	\end{defn}
	\begin{rem}\label{rem: decomp of 2d functions}
		Note that any function $f:[0,T]^2\rightarrow E$ can be decomposed into two functions $f=z+y$ where $y$ is zero on the boundary $\partial [0,T]^2:=\{0\}\times [0,T]\cup [0,T]\times \{0\}$ and for any $(s,t)\in [0,T]^2$ $\square_{s,t}z=0$. Indeed, by simple addition and subtraction, we see that 
		\begin{equation*}
			f(t_1,t_2)= \square_{0,t}f +f(t_1,0)+f(0,t_2)-f(0,0), 
		\end{equation*}
		thus by defining 
		\begin{equation*}
			z(t_1,t_2):=f(t_1,0)+f(0,t_2)-f(0,0)\quad \mathrm{and} \quad y(t_1,t_2)=\square_{0,t} f,
		\end{equation*}
		we see that $z$ and $y$ satisfy the claimed properties. Furthermore, considering the 2D-H\"older semi-norm  of $f$ over $[0,T]^2$, we see that 
		\begin{equation}\label{eq:decomp of holder norm}
			[f]_\alpha\sim_T[z]_{(1,0),\alpha}+[z]_{(0,1),\alpha}+[y]_{(1,1),\alpha}. 
		\end{equation}
		This decomposition and relation will play a central role in subsequent sections. 
	\end{rem}

	Let us also recall the definition of Besov
	spaces which will be of use towards the formulation of our regularization by noise results. For a more extensive introduction we refer to \cite{BahCheDan}. We will denote by $\mathscr{S}$ (respectively
	$\mathscr{S}^{\prime}$) the space of Schwartz functions on $\mathbb R^d$ (respectively its dual, the space of tempered distributions). For $f\in\mathscr{S}^{\prime}$
	we denote the Fourier transform by $\hat{f}=\mathscr{F}\left(f\right) = \int_{\mathbb R^d}e^{-i x \cdot} f(x)  dx$, where the integral notation is formal, with inverse $\mathscr{F}^{-1} f = (2\pi)^{-d}\int_{\mathbb R^d} e^{i z \cdot} \hat f(z)dz$.
	\begin{defn}
		\label{def:Dyadic partition of unity}Let $\chi,\rho\in C^{\infty}(\mathbb{R}^{d})$
		be two radial functions such that $\chi$ is supported on a ball $\mathcal{B}=\left\{ |x|\leq c\right\} $
		and $\rho$ is supported on an annulus $\mathcal{A}=\left\{ a\leq|x|\leq b\right\} $
		for $a,b,c>0$, such that 
		\begin{align*}
			\chi+\sum_{j\geq0}\rho\left(2^{-j}\cdot\right) & \equiv1,\\
			{\rm supp}\left(\chi\right)\cap{\rm supp}\left(\rho\left(2^{-j}\cdot\right)\right) & =\emptyset,\,\,\,\forall j\ge1,\\
			{\rm supp}\left(\rho\left(2^{-j}\cdot\right)\right)\cap{\rm supp}\left(\rho\left(2^{-i}\cdot\right)\right) & =\emptyset,\,\,\,\forall|i-j|\geq1.
		\end{align*}
		Then we call the pair $\left(\chi,\rho\right)$ a \emph{dyadic partition
			of unity}. Furthermore, we write $\rho_{j} :=\rho(2^{-j} \cdot)$
		for $j\geq 0$ and $\rho_{-1}=\chi$, as well as $K_{j}=\mathscr{F}^{-1}\rho_{j}$.
	\end{defn}
	
	The existence of a partition of unity is shown for example in \cite[Proposition 2.10]{BahCheDan}. We fix a partition of unity $(\chi,\rho)$ for the rest of the paper.
	
	\begin{defn}
		\label{def:Paley littlewood block}For $f\in\mathcal{\mathscr{S}}^{\prime}$
		we define its Littlewood-Paley blocks by
		\[
		\Delta_{j}f :=\mathscr{F}^{-1}(\rho_{j}\hat{f}) = K_j \ast f.
		\]
		It follows that $f=\sum_{j\geq-1}\Delta_{j}f$ with convergence in $\mathscr S'$.
	\end{defn}

	\begin{defn}
		\label{def:Besov space}For any $\alpha\in\mathbb{R}$ and $p,q\in\left[1,\infty\right]$,
		the \emph{Besov space} $B_{p,q}^{\alpha}(\RR^d)$ is
		\[
		B_{p,q}^{\alpha}(\RR^d):=\left\{ f\in\mathscr{S}^{\prime}\left|\|f\|_{B_{p,q}^{\alpha}(\RR^d)} :=\left(\sum_{j\geq-1}\left(2^{j\alpha}\| \Delta_{j}f \|_{L^{p}}\right)^{q}\right)^{\frac{1}{q}}<\infty\right.\right\},
		\]
		with the usual interpretation as $\ell^\infty$ norm if $q = \infty$. 
	\end{defn}

	At various places we will write $B_{p,q}^{\alpha}$ instead $B_{p,q}^{\alpha}(\RR^d)$ to simplify notation. Furthermore we will work with the classical space of  global H\"older continuous functions over the whole space $\RR^d$. We denote the space of globally bounded H\"older continuous functions from $\RR^d$ by $C^\alpha_x:=C^\alpha_b(\RR^d)$. We extend these spaces to the differentiable functions with H\"older continuous derivatives in the canonical way. Note that $B^\alpha_{\infty,\infty}(\RR^d) \simeq C^\alpha_b(\RR^d)$ whenever $\alpha$ is a positive non-integer number.

	\section{2D Non-linear Young integrals and equations}\label{sec:2DNLY}
	In this section we will provide a framework for a 2D non-linear Young theory, starting with the formulation of the 2D Sewing Lemma from \cite{Harang2020} and followed by non-linear Young integrals and equations. 
	
	\subsection{The 2D Sewing Lemma}
	
	In order to formulate the 2D Sewing Lemma, we will introduce  an extension of the familiar $\delta$ operator known from the theory of rough paths \cite{FriHai}. We  define this as follows: for a function $f:[0,T]^4\rightarrow \RR^d$, and $s<u<t\in [0,T]^2$ define 
	\begin{equation*}
		\begin{aligned}
			\delta^1_u f_{s,t}=\delta^1 f_{s,u,t} = f_{s,t}-f_{s,(u_1,t_2)}-f_{(u_1,s_2),t},
			\\
			\delta^2_u f_{s,t}=\delta^2 f_{s,u,t}=f_{s,t}-f_{s,(t_1,u_2)}-f_{(s_1,u_2),t}.
		\end{aligned}
	\end{equation*}
	Thus, for $i=1,2,$ $\delta^i:[0,T]\rightarrow L(C([0,T]^4); C([0,T]^5))$ (where $L(X;Y)$ is the space of linear operators from $X$ to $Y$). 
	Furthermore, we will invoke the composition of $\delta^1\circ \delta^2$ defined in the canonical way; i.e. $\delta^1\circ \delta^2 f=\delta^1(\delta^2 f)=\delta^2(\delta^1 f)$, and we note that $\delta^1\circ \delta^2:[0,T]^2\rightarrow L(C([0,T]^4); C([0,T]^6))$.  If we want to specify the variable in the $\delta$ operator then we will write $\delta_{u_i}^i,i =1,2$.

	\begin{defn}
		Consider $\alpha\in (0,1)^2$ and $\beta\in (1,\infty)^2$. We denote by $\cC^{\alpha,\beta}_2$ the space of all functions  $\Xi:[0,T]^4\rightarrow \RR^d$ such that $\Xi_{s,t}=0$ if $s_1=t_1$ or $s_2 = t_2$ and 
		\begin{equation*}
			[\Xi]_{\alpha,\beta}:=[\Xi]_{\alpha}+[\Xi]_{(0,1),\alpha,\beta}+[\Xi]_{(1,0),\alpha,\beta}+[\Xi]_{(1,1),\alpha,\beta}<\infty, 
		\end{equation*}
		where the semi-norm $[\Xi]_{\alpha}$, and the remaining terms above are given by 
		\begin{equation}\label{eq:holder norms}
			\begin{aligned}
				\,[\Xi]_{\alpha}&:=\sup_{s\neq t\in [0,T]^2} \frac{|\Xi_{s,t}|}{m(t-s)^\alpha},
				\\
				\,[\Xi]_{(1,0),\alpha,\beta}&:=\sup_{s<u< t\in [0,T]^2} \frac{|\delta^1 \Xi_{s,u,t}|}{|t_1-s_1|^{\beta_1}|t_2-s_2|^{\alpha_2}},
				\\
				[\Xi]_{(0,1),\alpha,\beta}&:=\sup_{s<u< t\in [0,T]^2} \frac{|\delta^2 \Xi_{s,u,t}|}{|t_1-s_1|^{\alpha_1}|t_2-s_2|^{\beta_2}},
				\\
				[\Xi]_{(1,1),\alpha,\beta}&:=\sup_{s<u< t\in [0,T]^2} \frac{|\delta^1\circ \delta^2 \Xi_{s,u,t}|}{m(t-s)^\beta}, 
			\end{aligned}
		\end{equation}
		where we recall that $m(t-s)^\beta$ is defined as in \eqref{eq:dist multi}. For later notational convenience we also define $[\delta \Xi]_{\alpha,\beta}=[\Xi]_{(0,1),\alpha,\beta}+[\Xi]_{(1,0),\alpha,\beta}+[\Xi]_{(1,1),\alpha,\beta}$. 
	\end{defn}
	
	When working with the two dimensional sewing lemma it is convenient to simplify notation for two dimensional partitions. We therefore provide the following definition. 
	\begin{defn}\label{def: partition}
		We will say that $\cP$ is a partition of the rectangle $[s,t]\subset [0,T]^2$ if 
		\begin{equation*}
			\cP=\cP^1\times \cP^2, 
		\end{equation*}
		where $\cP^1$ is a standard partition of $[s_1,t_1]$ and $\cP^2$ is a standard partition of $[s_2,t_2]$.
	\end{defn}
	
	We are now ready to state a two dimensional version of the sewing lemma. A version of this lemma was first introduced in \cite{towghi} using variation norms, and extended to the hyper-cubes in arbitrary dimension in the setting of H\"older type norms in \cite{Harang2020}. Here
	we follow the last reference. 
	
	\begin{lem}[\cite{Harang2020}, Lemma 14]\label{lem:2d sewing}
		For $\alpha\in (0,1)^2$ and $\beta\in (1,\infty)^2$, let $\Xi\in \cC^{\alpha,\beta}_2$. Let $\cP:=\cP[s,t]$ denote a partition of $[s,t]\subset [0,T]^2$ in the sense of Definition \ref{def: partition}.   
		There exists a unique continuous linear functional $\cI:\cC^{\alpha,\beta}_2 \rightarrow C^\alpha_t$  given by
		\begin{equation*}
			\cI(\Xi)_{[s,t]} :=\lim_{|\cP|\rightarrow 0} \sum_{[u,v]\in \cP} \Xi_{u,v},  
		\end{equation*}
		where the limit is taken over any sequence of partitions with $|\cP|\rightarrow 0$.  We note that it is under the restriction $t\mapsto \cI(\Xi)_t:=\cI(\Xi)_{[0,t]}$ that we have $\cI(\Xi)\in C^\alpha_t$,  and we have that  $\square_{s,t}\cI(\Xi)_\cdot=\cI(\Xi)_{[s,t]}$. 
		Furthermore, there exists a constant $C=C(\alpha,\beta,T)>0$ such that the function $\cI(\Xi)$ satisfies the following inequality \begin{equation}\label{eq:sew lem ineq}
			|\cI(\Xi)_{[s,t]}-\Xi_{s,t}|\leq C m(t-s)^\alpha |t-s|^{\beta-\alpha} [\delta \Xi]_{\alpha,\beta}, 
		\end{equation}
		where $m(t-s)^\alpha$ and $|t-s|^{\beta-\alpha}$ are defined in \eqref{eq:dist multi} and \eqref{eq:dist}.
	\end{lem}
	
	For the sake of brevity, we refer the reader to \cite[Lem. 14]{Harang2020} for a full proof of this lemma.

	\subsection{2D non-linear Young integral}
	With the aim of constructing a 2D analogue of the non-linear Young integral, we will need to control the rectangular increments of differentiable non-linear functions. We therefore provide the following elementary lemma, which will be frequently used in the sequel. 
	
	\begin{lem}\label{Lem: reg of f}
		Let $f\in C^{1+\eta}(\RR^d)$ for some $\eta\in (0,1)$. Then the following bound holds: For all $x,y,z,w\in \RR^d$ 
		\begin{equation}\label{eq:control of rectangular increment}
			|f(x)-f(y)-f(z)+f(w)|\leq \|f\|_{C^{1+\eta}}(|x-y-z+w|+|x-y|(|x-z|+|y-w|)^\eta).
		\end{equation}
	\end{lem}
	\begin{proof}
		From a first order Taylor expansion, it follows that 
		\begin{multline}\label{eq:first order taylor}
			f(x)-f(y)-f(z)+f(w)=\int_0^1\nabla f(\theta x+(1-\theta)y)-\nabla f(\theta z+(1-\theta)w)\dd \theta (x-y)
			\\
			+\int_0^1\nabla f(\theta x+(1-\theta)y)\dd \theta (x-y-z+w). 
		\end{multline}
		Using that $\nabla f\in C^\eta$, it follows that 
		\begin{equation*}
			|f(x)-f(y)-f(z)+f(w)|\lesssim  \|f\|_{C^{1+\eta}}(|x-y-z+w|+\int_0^1|\theta (x-z)+(1-\theta)(y-w)|^\eta \dd \theta\,|x-y|).
		\end{equation*}
	\end{proof}
	
	With the above lemma at hand, we are now ready to prove the existence of the 2D Non-Linear Young integral (NLY), and its properties. 
	
	\begin{prop}[2D-Non-linear Young integral]\label{prop. 2d NLY integral}
		Let $A:[0,T]^2\times \RR^d \rightarrow \RR^d$ be such that for some $\gamma\in (\frac{1}{2},1]^2$ and some $\eta\in (0,1]$, $A\in C^\gamma_tC^{1+\eta}_{x}$.  
		Consider a path $y\in C^\alpha_t([0,T]^2;\RR^d)$ with $\alpha\in(0,1)^2$ such that for $i=1,2$ we have $\alpha_i  \eta +\gamma_i>1$. 
		Then the 2D non-linear Young integral of $y$ with respect to $A$ is defined by 
		\begin{equation}\label{eq:2dnly}
			\int_0^t A(\dd s,y_s) :=\lim_{|\cP|\rightarrow 0} \sum_{[u,v]\in \cP} \square_{u,v}A(y_u),
		\end{equation}
		where $\cP$ is a partition of $[0,t]\subset [0,T]^2$, as given in Definition \ref{def: partition}. 
		Furthermore, there exists a constant $C>0$ such that  
		\begin{equation}\label{eq:bound NLY integral}
			|\int_s^t A(\dd s,y_s)-\square_{s,t}A(y_s)|\leq C\|A\|_{C^{\gamma}_tC^{1+\eta}_x} ([y]_{(1,1),\alpha}\vee ([y]_{(1,0),\alpha}+[y]_{(0,1),\alpha})^{1+\eta}) m(t-s)^\gamma |t-s|^{\eta\alpha},
		\end{equation}
		and it follows that 
		\begin{equation*}
			[0,T]^2\ni t\mapsto \int_0^t A(\dd s,y_s)\in C^\gamma_t.
		\end{equation*}
	\end{prop}
	\begin{proof}
		Towards the construction of the integral in \eqref{eq:2dnly}, we will apply the 2D sewing lemma to the integrand $\square_{u,v}A(y_u)$. Thus, we need to check that $[0,T]^4\ni(s,t)\mapsto  \square_{s,t}A(y_s)$ belongs to  $\cC^{\alpha,\beta}_2$ for the given $\alpha$ and some well chosen $\beta\in (1,\infty)^2$. Let $u=(u_1,u_2)$ such that $s<u<t$. 
		It is readily checked that 
		\begin{equation}\label{eq:delta 1 on A}
			\delta^1_{u_1}\square_{s,t}A(y_s)=\square_{(u_1, s_2), t}A(y_s)-\square_{(u_1, s_2), t}A(y_{(u_1, s_2)}),
		\end{equation}
		and similarly for $\delta^2_{u_2} \square_{s,t}A(y_s)$, and we have 
		\begin{equation}\label{eq:delta12 on A}
			\delta^1_{u_1}\circ \delta^2_{u_2}\square_{s,t}A(y_s) = -\left(\square_{u,t}A(y_u)-\square_{u,t}A(y_{u_1,s_2})-\square_{u,t}A(y_{s_1,u_2})+\square_{u,t}A(y_s)\right).
		\end{equation}
		We first prove the necessary regularity of the increment in \eqref{eq:delta 1 on A}, and a similar estimate for $\delta^2$ follows directly. 
		Using that $ A\in C_t^\gamma C^{1+\eta}_{x}$ we have that \begin{equation*}
			|\delta^1_{u_1}\square_{s,t}A(y_s)|\leq \|A_{(u_1,s_2),t}\|_{C^1_x}[y]_{(1,0),\alpha}|u_1-s_1|^{\alpha_1}. 
		\end{equation*}
		Invoking the assumption that $t\mapsto A(t,\cdot)\in C^\gamma_t$ we get that  
		\begin{equation}\label{eq:est on delta 1}
			|\delta^1_{u_1}\square_{s,t}A(y_s)|\leq \|A\|_{C^\gamma_tC^1_x}[y]_{(1,0),\alpha}|t_1-s_1|^{\alpha_1+\gamma_1}|t_2-s_2|^{\gamma_2}, 
		\end{equation}
		where we have used that $|t_1-u_1|^{\gamma_1}|u_1-s_1|^{\alpha_1}\leq |t_1-s_1|^{\alpha_1+\gamma_1}$ since $s<u<t$. 
		Now, we will consider the increment in \eqref{eq:delta12 on A}.
		By Lemma \ref{Lem: reg of f} and \eqref{eq:delta12 on A} it follows that 
		\begin{equation}\label{eq: first bound A delta 12}
			|\delta^1_{u_1}\circ \delta^2_{u_2}\square_{s,t}A(y_s)|\leq \|\square_{u,t} A \|_{C^{1+\eta}_x}\left(|\square_{s,u}y|+|y_{s_1,s_2}-y_{u_1,s_2}|(|y_{s_1,s_2}-y_{s_1,u_2}|+|y_{u_1,s_2}-y_{u_1,u_2}|)^\eta\right).
		\end{equation}
		Invoking the assumption of time regularity of $A$ and $y$, we see that 
		\begin{multline}\label{eq:est on delta 12}
			|\delta^1_{u_1}\circ \delta^2_{u_2}\square_{s,t}A(y_s)|\lesssim_T \|A\|_{C^\gamma_t C^{1+\eta}_x}([y]_{(1,1),\alpha}+[y]_{(1,0),\alpha}([y]_{(1,0),\alpha}+[y]_{(0,1),\alpha})^\eta) 
			\\
			m(t-s)^{\gamma}(m(t-s)^\alpha+|t_1-s_1|^{\alpha_1}|t_2-s_2|^{\eta \alpha_2}).
		\end{multline}
		Here we recall that $m(t-s)^\alpha=|t_1-s_1|^{\alpha_1}|t_2-s_2|^{\alpha_2}$.
		Note that 
		\begin{equation*}
			m(t-s)^{\gamma}(m(t-s)^\alpha+|t_1-s_1|^{\alpha_1}|t_2-s_2|^{\eta \alpha_2}) \lesssim_T m(t-s)^{\gamma+\alpha\eta}. 
		\end{equation*}
		Thus, from the estimates in \eqref{eq:est on delta 1} and \eqref{eq:est on delta 12} and using the assumption that $\alpha\eta+\gamma>1$, we define  $\beta=\alpha\eta+\gamma$ and it follows that $\Xi_{s,t}:=\square_{s,t}A(y_s)$ is contained in $\cC^{\alpha,\beta}_2$. 
		Using that $[y]_{(1,0),\alpha}\leq [y]_{(1,0),\alpha}+[y]_{(0,1),\alpha}$ and  that for positive numbers $a,b$ we have  $a+b\leq 2a\vee b$, we have that  
		\begin{equation*}
			([y]_{(1,1),\alpha}+[y]_{(1,0),\alpha}([y]_{(1,0),\alpha}+[y]_{(0,1),\alpha})^\eta) \leq 2 \left([y]_{(1,1),\alpha}\vee ([y]_{(1,0),\alpha}+[y]_{(0,1),\alpha})^{1+\eta}\right), 
		\end{equation*}
		we conclude by an application of Lemma  \ref{lem:2d sewing}, where the inequality in \eqref{eq:bound NLY integral} follows directly from \eqref{eq:sew lem ineq}. 
		
	\end{proof}
	
	\begin{rem}
		\label{difference 1d2d integral}
		Note that our construction of the 2D non-linear Young integral above requires one additional degree of spatial regularity compared with the 1d-nonlinear Young setting (refer e.g. to \cite[Theorem 2.7]{Galeati2020NonlinearYD}). This more restrictive condition appears, as we have to control already at this step four-point increments in the form of \eqref{eq:delta12 on A} due to the necessity of controlling $\delta^1\circ\delta^2 \square A$.
	\end{rem}
	The following Lemma establishes the consistency of the 2D non-linear Young integral constructed in the above Proposition \ref{prop. 2d NLY integral} with respect to classical Riemann integration in the setting of continuously differentiable $A$. 
	\begin{lem}
		Suppose the conditions of Proposition \ref{prop. 2d NLY integral} hold. Suppose moreover $\partial_{t_1}\partial_{t_2}A$ exists and is continuous. Then the  2D non-linear Young integral constructed in Proposition \ref{prop. 2d NLY integral} coincides with the corresponding Riemann integral, i.e. 
		\[
		\int_0^t A(\dd s, y_s)=\int_0^t \partial_{t_1}\partial_{t_2}A(s, y_s)\dd s
		\]
		\label{consistency 2d nonlinear young}
	\end{lem}
	\begin{proof}
		Remark that as we have 
		\[
		\square_{s,t}A(y_s)=\int_s^t \partial_{t_1}\partial_{t_2}A(r, y_s)\dd r
		\]
		it suffice to show that for any sequence $\mathcal{P}^n$ of rectangular partitions of $[s_1, t_1]\times [s_2, t_2]$ such that $|\mathcal{P}^n|\to 0$, we have that
		\begin{equation*}
			D_n= \sum_{[u_1,v_1]\times [u_2,v_2]\in \cP^n} [\int_u^v \partial_{t_1}\partial_{t_2}A(r, y_u)- \partial_{t_1}\partial_{t_2}A(r, y_r)\dd r] \rightarrow 0 \quad as\quad  n\rightarrow \infty.. 
		\end{equation*}
		It is readily seen that
		\begin{equation*}
			|D_n|\leq \sup_{[u_1,v_1]\times [u_2,v_2]\in \cP^n} \sup_{r\in [u, v]}|\partial_{t_1}\partial_{t_2}A(r, y_s)- \partial_{t_1}\partial_{t_2}A(r, y_r)|m(t-s), 
		\end{equation*}
		Since the mesh size of $\cP^n$ goes to $0$ as $n\rightarrow \infty$ by assumption, it follows that 
		$$\sup_{[u_1,v_1]\times [u_2,v_2]\in \cP^n} \sup_{r\in [u, v]}|\partial_{t_1}\partial_{t_2}A(r, y_s)- \partial_{t_1}\partial_{t_2}A(r, y_r)|\rightarrow 0,$$
		as $n\rightarrow \infty$ since $\partial_{t_1}\partial_{t_2}A$ is assumed to be continuous. 
		We conclude that $D_n\rightarrow 0$ as $n\rightarrow \infty$ which concludes the proof. 
	\end{proof}

	Being a concept of integration constructed free of probability, the following stability estimates are not only useful for the subsequent proof of existence and uniqueness of non-linear Young equations, but also provides powerful estimates when applied in combination with stochastic processes, as will be evident in our application towards regularization by noise.

	\begin{prop}[Stability of integral]\label{prop:stability}
		For some $\gamma\in (\frac{1}{2},1]^2$ and $\eta\in (0,1]$, consider two functions  $A,\tilde{A}\in C^\gamma_t C^{2+\eta}_x$. Furthermore, suppose $y,\tilde{y}\in C^\alpha_t$ such that $\alpha  \eta +\gamma>1$ 
		Then the 2D non-linear Young integral satisfies the following inequality 
		\begin{equation}\label{eq:bound diff NLY integrals}
			|\int_s^t A(\dd r,y_r)-\int_s^t \tilde{A}(\dd r, \tilde{y}_r)| \leq\left( C_1\|A-\tilde{A}\|_{C^\gamma_tC^{2+\eta}_x}+C_2(|y_s-\tilde{y}_s|+[y-\tilde{y}]_{\alpha;[s,t]})\right)m(t-s)^\gamma,
		\end{equation}
		where the constants $C_1$ and $C_2$ are given by 
		\begin{equation}\label{eq:C const}
			\begin{aligned}
				C_1=&K ([y]_\alpha\vee [\tilde{y}]_\alpha),
				\\
				C_2=&K(\|A\|_{C^\gamma_tC^{2+\eta}_x}\vee\|\tilde{A}\|_{C^\gamma_tC^{2+\eta}_x})
				\\
				&([y]_{(1,1),\alpha}+[\tilde{y}]_{(1,1),\alpha})
				\vee ([y]_{(0,1),\alpha}+[y]_{(1,0),\alpha}+[\tilde{y}]_{(0,1),\alpha}+[\tilde{y}]_{(1,0),\alpha})^{1+\eta},
			\end{aligned}
		\end{equation}
		for some constant $K$ depending on $T,\alpha,\gamma,\eta$. 
	\end{prop}
	\begin{proof}
		To prove this, we will apply Lemma \ref{lem:2d sewing} to the increment   $\Xi_{s,t}:=\square_{s,t}A(y_s)-\square_{s,t}\tilde{A}(\tilde{y}_s)$, in order invoke the inequality \eqref{eq:sew lem ineq}. 
		We have $\Xi\in C^\gamma_t$, as we get that
		\begin{equation}\label{eq:bound xi}
			|\Xi_{s,t}|\lesssim_T (\|A\|_{C^\gamma_tC^{0}_x}+\|\tilde{A}\|_{C^\gamma_tC^{0}_x})m(t-s)^\gamma, 
		\end{equation}
		We proceed to  first  prove bounds for $\delta^i\Xi$ for $i=1,2$ and then we will provide bounds for $\delta^1\circ \delta^2\Xi$.
		
		For $u=(u_1,u_2)$ we observe that 
		\begin{equation*}
			\delta_{u_1}^1 \Xi_{s,t} =\square_{(u_1,s_2),t}A(y_s)-\square_{(u_1,s_2),t}A(y_{(u_1,s_2)})-
			\left[\square_{(u_1,s_2),t}\tilde{A}(\tilde{y}_s)-\square_{(u_1,s_2),t}\tilde{A}(\tilde{y}_{(u_1,s_2)})\right]. 
		\end{equation*}
		Define the function $G:[0,T]^2\times \RR^d\times \RR^d\rightarrow \RR^{d\times d}$ by \[G_{t}(x, \tilde{x}):=\int_0^1 DA_{t}(\theta x +(1-\theta)\tilde{x})\dd \theta ,  
		\]
		where $DA$ is the matrix valued derivative of $A$.  
		Similarly we define $\tilde{G}$ from the composition of $\tilde{y}$ and $\tilde{A}$.
		Note in particular that due to the assumption that $x\mapsto A(t,x)\in C^{2+\eta}_x$ for all $t\in [0,T]^2$, we have that 
		\begin{equation}\label{eq:G bound}
			|G_t(x, \tilde{x})|\leq\|A_t(\cdot)\|_{C^{1}_x}, \qquad   |G_t(x, \tilde{x})-G_t(z, \tilde{z})\leq |\|A_t(\cdot)\|_{C^{2}_x}(|x-\tilde
			x-z+\tilde{z}|+|\tilde{x}-\tilde{z}|)
		\end{equation}
		and similarly for $\tilde{G}$. 
		Then 
		\begin{equation*}
			\delta_{u_1}^1 \Xi_{s,t}=\square_{(u_1,s_2),t}G(y_s,y_{(u_1,s_2)})(y_s-y_{(u_1,s_2)})-\square_{(u_1,s_2),t}\tilde{G}(\tilde{y}_s,\tilde{y}_{(u_1,s_2)})(\tilde{y}_s-\tilde{y}_{(u_1,s_2)}).
		\end{equation*}
		By addition and subtraction of the term $\square_{(u_1,s_2),t}G(y_s,y_{(u_1,s_2)})(\tilde{y}_s-\tilde{y}_{(u_1,s_2)})$ in the above relation we will seek to control the following two terms
		\begin{align*}
			I_1(s,u,t) &= \square_{(u_1,s_2),t}G(y_s,y_{(u_1,s_2)})\left(y_s-y_{(u_1,s_2)}-\tilde{y}_s+\tilde{y}_{(u_1,s_2)}\right),
			\\
			I_2(s,u,t) &=  \left(\square_{(u_1,s_2),t}G(y_s,y_{(u_1,s_2)})-\square_{(u_1,s_2),t}\tilde{G}
			(\tilde{y}_s,\tilde{y}_{(u_1,s_2)})\right)(\tilde{y}_s-\tilde{y}_{(u_1,s_2)}). 
		\end{align*}
		For the term $I_1$, using \eqref{eq:G bound} it is readily checked that 
		\begin{equation*}
			|I_1(s,u,t)|\leq \|A\|_{C^\gamma_tC^1_x}[y-\tilde{y}]_{(1,0),\alpha;[s,t]}|t_1-s_1|^{\alpha_1+\gamma_1}|t_2-s_2|^{\gamma_2}. 
		\end{equation*}
		Next, we consider  $I_2$. To this end, first observe that by addition and subtraction of the term  $\square_{(u_1,s_2),t}G(\tilde{y}_s,\tilde{y}_{(u_1,s_2)})$ we have 
		\begin{multline*}
			\square_{(u_1,s_2),t}G(y_s,y_{(u_1,s_2)})-\square_{(u_1,s_2),t}\tilde{G}(\tilde{y}_s,\tilde{y}_{(u_1,s_2)})  
			\\
			= \left[\square_{(u_1,s_2),t}G(y_s,y_{(u_1,s_2)})-\square_{(u_1,s_2),t}G(\tilde{y}_s,\tilde{y}_{(u_1,s_2)})\right]+\left[\square_{(u_1,s_2),t}G(\tilde{y}_s,\tilde{y}_{(u_1,s_2)})-\square_{(u_1,s_2),t}\tilde{G}(\tilde{y}_s,\tilde{y}_{(u_1,s_2)})\right]. 
		\end{multline*}
		For the first bracket on the right hand side above,
		we use the second inequality in \eqref{eq:G bound} to see that
		\begin{equation*}
			|\square_{(u_1,s_2),t}G(y_s,y_{(u_1,s_2)})-\square_{(u_1,s_2),t}G(\tilde{y}_s,\tilde{y}_{(u_1,s_2)})|\leq 2 \|A\|_{C^\gamma_tC^2_x}\|y-\tilde{y}\|_{\infty;[s,t]} m(t-s)^\gamma.  
		\end{equation*}
		For the term in the second bracket, we use the first inequality in \eqref{eq:G bound} to observe that 
		\begin{align*}
			|\square_{(u_1,s_2),t}G(\tilde{y}_s,\tilde{y}_{(u_1,s_2)})-\square_{(u_1,s_2),t}\tilde{G}(\tilde{y}_s,\tilde{y}_{(u_1,s_2)})| &\leq \|A-\tilde{A}\|_{C^\gamma_tC^1_x}m(t-s)^\gamma. 
		\end{align*}
		Combining the above estimates, we see that 
		\begin{equation*}
			|I_2(s,u,t)|\lesssim [\tilde{y}]_{(1,0),\alpha;[s,t]}(\|A\|_{C^\gamma_tC^2_x}\|y-\tilde{y}\|_{\infty;[s,t]}+\|A-\tilde{A}\|_{C^\gamma_tC^2_x}) |t_1-s_1|^{\alpha_1+\gamma_1}|t_2-s_2|^{\gamma_2}. 
		\end{equation*}
		From a combination of our estimates for $I_1$ and $I_2$, we see that 
		\begin{equation*}
			|\delta^1_{u_1}\Xi_{s,t}|\lesssim (1+\|A\|_{C^\gamma_tC^2_x}+[\tilde{y}]_{(1,0),\alpha;[s,t]})(|y_s-\tilde{y}_s|+[y-\tilde{y}]_{\alpha;[s,t]}+\|A-\tilde{A}\|_{C^\gamma_tC^2_x}) |t_1-s_1|^{\alpha_1+\gamma_1}|t_2-s_2|^{\gamma_2}. 
		\end{equation*}
		By similar considerations, we can show that 
		\begin{equation*}
			|\delta^2_{u_2}\Xi_{s,t}|\lesssim (1+\|A\|_{C^\gamma_tC^2_x}+[\tilde{y}]_{(0,1),\alpha;[s,t]})(|y_s-\tilde{y}_s|+[y-\tilde{y}]_{\alpha;[s,t]}+\|A-\tilde{A}\|_{C^\gamma_tC^2_x}) |t_1-s_1|^{\gamma_1}|t_2-s_2|^{\alpha_2+\gamma_2}. 
		\end{equation*}
		Next we move on to consider the term $\delta^1\circ \delta^2 \Xi$.
		By linearity of the $\delta$-operators, we have similarly to \eqref{eq:delta12 on A} that 
		\begin{equation*}
			\begin{split}
				\delta^1_{u_1}\circ \delta^2_{u_2} \Xi_{s,t} &= -\left(\square_{u,t}A(y_u)-\square_{u,t}A(y_{u_1,s_2})-\square_{u,t}A(y_{s_1,u_2})+\square_{u,t}A(y_s)\right)\\
				&+\left(\square_{u,t}\tilde{A}(\tilde{y}_u)-\square_{u,t}\tilde{A}(\tilde{y}_{u_1,s_2})-\square_{u,t}\tilde{A}(\tilde{y}_{s_1,u_2})+\square_{u,t}\tilde{A}(\tilde{y}_s)\right). 
			\end{split}
		\end{equation*}
		By addition and subtraction of $\tilde{A}(y)$, we see that this is the same as 
		\begin{equation}\label{eq:del 1 2 relation}
			\begin{split}
				\delta^1_{u_1}\circ \delta^2_{u_2} \Xi_{s,t} &= -\left(\square_{u,t}(A-\tilde{A})(y_u)-\square_{u,t}(A-\tilde{A})(y_{u_1,s_2})-\square_{u,t}(A-\tilde{A})(y_{s_1,u_2})+\square_{u,t}(A-\tilde{A})(y_s)\right)\\
				&+\Big(\square_{u,t}\tilde{A}(\tilde{y}_u)-\square_{u,t}\tilde{A}(y_u)-\square_{u,t}\tilde{A}(\tilde{y}_{u_1,s_2})+\square_{u,t}\tilde{A}(y_{u_1,s_2})\\
				&-\square_{u,t}\tilde{A}(\tilde{y}_{s_1,u_2})+\square_{u,t}\tilde{A}(y_{s_1,u_2})+\square_{u,t}\tilde{A}(\tilde{y}_s)-\square_{u,t}\tilde{A}(y_s)\Big)
			\end{split}
		\end{equation}
		For the first term on the right hand side, we see that by defining a new function $\bar{A}=A-\tilde{A}$, we can easily use the same arguments as in the proof of Proposition \ref{prop. 2d NLY integral} (see in particular \eqref{eq:est on delta 12}) to see that 
		\begin{equation}\label{eq:bound first term A-A}
			\begin{split}
				&|\square_{u,t}(A-\tilde{A})(y_u)-\square_{u,t}(A-\tilde{A})(y_{u_1,s_2})-\square_{u,t}(A-\tilde{A})(y_{s_1,u_2})+\square_{u,t}(A-\tilde{A})(y_s)|\\
				&\lesssim \|A-\tilde{A}\|_{C^\gamma_tC^{2}_x}([y]_{\alpha;[s,t]}\vee [y]_{\alpha;[s,t]}^{2}) m(t-s)^{\alpha+\gamma}. 
			\end{split}
		\end{equation}
		Note that in contrast to \eqref{eq:est on delta 12} we have $A,\tilde{A}\in C^{2+\eta}$ and thus the $\eta$ dependence on the right hand side above disappear. 
		The remaining term in \eqref{eq:del 1 2 relation} must be treated differently, and in a similar procedure as what we did for $\delta^1 \Xi$. By a first order Taylor approximation, we see that for $x, \tilde{x}\in \RR^d$
		\begin{equation*}
			\square_{u,t}\tilde{A}(x)-\square_{u,t} \tilde{A}(\tilde{x})=\square_{u,t}\tilde{G}(x,\tilde{x})(x-\tilde{x}). 
		\end{equation*}
		In order to control the right hand side of \eqref{eq:del 1 2 relation} it now remains  to bound the term 
		\begin{equation*}
			\begin{split}
				&+\Big(\square_{u,t}\tilde{A}(\tilde{y}_u)-\square_{u,t}\tilde{A}(y_u)-\square_{u,t}\tilde{A}(\tilde{y}_{u_1,s_2})+\square_{u,t}\tilde{A}(y_{u_1,s_2})\\
				&-\square_{u,t}\tilde{A}(\tilde{y}_{s_1,u_2})+\square_{u,t}\tilde{A}(y_{s_1,u_2})+\square_{u,t}\tilde{A}(\tilde{y}_s)-\square_{u,t}\tilde{A}(y_s)\Big)\\
				&=\square_{u,t}\tilde{G}(\tilde{y}_u, y_u)(\tilde{y}_u-y_u)-
				\square_{u,t}\tilde{G}(\tilde{y}_{u_1, s_2}, y_{u_1, s_2})(\tilde{y}_{u_1, s_2}-y_{u_1, s_2})\\
				&-\square_{u,t}\tilde{G}(\tilde{y}_{s_1, u_2}, y_{s_1, u_2})(\tilde{y}_{s_1, u_2}-y_{s_1, u_2})+\square_{u,t}\tilde{G}(\tilde{y}_s, y_s)(\tilde{y}_s-y_s)\\
				&=:\square_{s,u}\left(\square_{u,t} \tilde{G}(y,\tilde{y})(y-\tilde{y})\right).
			\end{split}
		\end{equation*}
		Due to the multiplicative nature of $\square_{u,t}\tilde{G}(y,\tilde{y})(y-\tilde{y})$, we must be careful. From \cite[Lemma 5]{Harang2020} it follows that we have the decomposition
		\begin{equation*}
			\square_{s,u}\left(\square_{u,t} \tilde{G}(y,\tilde{y})(y-\tilde{y})\right)= \sum_{i=1}^3 J^i_{s,u,t}, 
		\end{equation*}
		where
		\begin{equation*}
			\begin{aligned}
				J^1_{s,u,t}&:=  \left[ \square_{s,u}\left(\square_{u,t} \tilde{G}(y,\tilde{y})\right)\right](y_u-\tilde{y}_u),
				\\
				J^2_{s,u,t}&:= \left[\square_{u,t}\tilde{G}(y_s,\tilde{y}_s)\right]\square_{s,u}(y-\tilde{y}),
				\\
				J^3_{s,u,t}&:=  [\square_{u,t}\tilde{G}(y_{u_1,s_2},\tilde{y}_{u_1,s_2})-\square_{u,t}\tilde{G}(y_{s_1,s_2},\tilde{y}_{s_1,s_2})][(y_{s_1,u_2}-\tilde{y}_{s_1,u_2})-(y_s-\tilde{y}_s)].
			\end{aligned}
		\end{equation*}
		Each of these terms must be treated separately. The simplest term is $J^2$, where we observe that 
		\begin{equation*}
			|J^2_{s,u,t}|\lesssim \|\tilde{G}\|_{C^\gamma_t L^\infty_x}[y-\tilde{y}]_{(1,1),\alpha;[s,t]}m(t-s)^{\alpha+\gamma}. 
		\end{equation*}
		Invoking again the first bound in \eqref{eq:G bound}, we see that 
		\begin{equation}\label{eq:bound J^2}
			|J^2_{s,u,t}|\lesssim \|\tilde{A}\|_{C^\gamma_t C^1_x}[y-\tilde{y}]_{(1,1),\alpha;[s,t]}m(t-s)^{\alpha+\gamma}.
		\end{equation}
		Next we consider $J^1$. By Lemma \ref{Lem: reg of f}, setting $z=(y,\tilde{y})$, we see that 
		\begin{equation*}
			\begin{split}
				| \square_{s,u}\left(\square_{u,t} \tilde{G}(y,\tilde{y})\right)|&\lesssim \|\square_{u,t}\tilde{G}\|_{C^{1+\eta}_x}[( [z]_{(1,1),\alpha;[s,t]})\vee([z]_{(0,1),\alpha;[s,t]}+[z]_{(1,0),\alpha;[s,t]})^{1+\eta}] m(t-s)^{\gamma+\eta\alpha}
				\\
				&\lesssim \|\tilde{A}\|_{C^\gamma_tC^{2+\eta}_x}[( [z]_{(1,1),\alpha;[s,t]})\vee([z]_{(0,1),\alpha;[s,t]}+[z]_{(1,0),\alpha;[s,t]})^{1+\eta}] m(t-s)^{\gamma+\eta\alpha}, 
			\end{split}
		\end{equation*}
		where we have used a slight extension of the second inequality in \eqref{eq:G bound} using Lemma \ref{Lem: reg of f}. Let us stress that it is precisely at this point that the required regularity of $A, \tilde{A}\in C^\gamma_tC^{2+\eta}_x$ enters into the picture, while previous esimates required less spatial regularity.  Note that 
		$[z]_\alpha\leq [y]_\alpha+[\tilde{y}]_\alpha$. 
		Thus for $J^1$ we get 
		\begin{multline}\label{eq:bound J1}
			|J^1_{s,u,t}| \lesssim  \|\tilde{A}\|_{C^\gamma_tC^{2+\eta}_x}([y]_{(1,1),\alpha;[s,t]}+[\tilde{y}]_{(1,1),\alpha;[s,t]})\\
			\vee ([y]_{(0,1),\alpha;[s,t]}+[y]_{(1,0),\alpha;[s,t]}+[\tilde{y}]_{(0,1),\alpha;[s,t]}+[\tilde{y}]_{(1,0),\alpha;[s,t]})^{1+\eta}] \\
			(|y_s-\tilde{y}_s|+[y-\tilde{y}]_{\alpha;[s,t]})m(t-s)^{\gamma+\eta\alpha}, 
		\end{multline}
		where we have used that $\|y-\tilde{y}\|_{\infty;[s,t]}\leq |y_s-\tilde{y}_s|+[y-\tilde{y}]_{\alpha;[s,t]}$. 
		At last we consider $J^3$, and by using the fact that $z\mapsto \tilde{G}(z)$ is differentiable, we get from similar type of estimates as above 
		\begin{equation}\label{eq:bound J3}
			|J^3_{s,u,t}|\lesssim \|\tilde{A}\|_{C^\gamma_tC^{2}_x}([y]_{\alpha;[s,t]}+[\tilde{y}]_{\alpha;[s,t]})[y-\tilde{y}]_{\alpha;[s,t]} m(t-s)^{\alpha+\gamma}. 
		\end{equation}
		Combining our bounds for $J^1,J^2$ and $J^3$ from \eqref{eq:bound J1}, \eqref{eq:bound J^2} \eqref{eq:bound J3}, we have that 
		\begin{multline}\label{eq:bound G term}
			| \square_{s,u}\left(\square_{u,t} \tilde{G}(y,\tilde{y})(y-\tilde{y})\right)|
			\lesssim_T  \|\tilde{A}\|_{C^\gamma_tC^{2+\eta}_x}[([y]_{(1,1),\alpha;[s,t]}+[\tilde{y}]_{(1,1),\alpha;[s,t]})
			\\
			\vee ([y]_{(0,1),\alpha;[s,t]}+[y]_{(1,0),\alpha;[s,t]}+[\tilde{y}]_{(0,1),\alpha;[s,t]}+[\tilde{y}]_{(1,0),\alpha;[s,t]})^{1+\eta}]  (|y_s-\tilde{y}_s|+[y-\tilde{y}]_{\alpha;[s,t]})m(t-s)^{\gamma+\eta\alpha}. 
		\end{multline}
		A bound for \eqref{eq:del 1 2 relation} now follows by a combination of \eqref{eq:bound G term} and \eqref{eq:bound first term A-A} and we obtain 
		\begin{equation*}
			| \delta_{u_1}^1\circ \delta^2_{u_2} \Xi_{s,t}|
			\\
			\leq  C_2 (\|A-\tilde{A}\|_{C^\gamma_tC^{2+\eta}_x} +|y_s-\tilde{y}_s|+[y-\tilde{y}]_{\alpha;[s,t]})m(t-s)^{\gamma+\eta\alpha},
		\end{equation*}
		where $C_2$ is given as in \eqref{eq:C const}. It now follows that we can apply the 2D sewing lemma (Lemma \ref{lem:2d sewing}), and invoking the inequality \eqref{eq:sew lem ineq} in this lemma, the bound in \eqref{eq:bound diff NLY integrals} follows. 
	\end{proof}
	
	\begin{rem}
		Again, we require one more degree of spatial regularity in our setting of stability of  2D non-linear Young integrals than in the 1d-nonlinear Young setting (compare e.g. to \cite[Theorem 2.7.4]{Galeati2020NonlinearYD} for the limit case $\delta=1$). Essentially, as stability estimates of the above form require one additional degree of spatial regularity compared with the regime of existence of the integral, the difference in spatial regularity constraints observed already in Remark \ref{difference 1d2d integral} carries over. 
	\end{rem}

	With the stability estimate for NLY integrals, we are now ready to also consider integral equations, where the 2D NLY integral appear. Similarly as how the field of (1D) non-linear Young equations is simply a generalisation of classical Young differential equations, the 2D non-linear Young equations are  generalizing the already established notion of 2D Young equations.
	
	In the following theorem we prove existence and uniqueness of these equations for sufficiently smooth non-linear functions $A:[0,T]^2\times \RR^d\rightarrow \RR ^d$.  As the techniques are strongly based on the proof of \cite[Theorem 25]{Harang2020}, we will be mostly concerned with various estimates that differ from this reference due to the non-linear Young structure of the integral.

	\begin{thm}[Existence and Uniqueness]\label{thm:nly equation}
		Suppose $A\in C^\gamma_t C^{2+\eta}_x$ for some $\gamma\in (\frac{1}{2},1]^2$ and $\eta \in (0,1)$. Furthermore, let $\xi\in C^\gamma_t$ be such that for any $s<t\in [0,T]^2$,  $\square_{s,t} \xi=0$. 
		If $(1+\eta)\gamma>1$, then there exists a unique solution $\theta\in C^\gamma_t$ to the equation 
		\begin{equation}\label{eq:2d nly eq}
			\theta_t=\xi_t+\int_0^t A(\dd s,\theta_s),\quad t\in [0,T]^2, 
		\end{equation}
		where the integral is interpreted in the sense of the NLY integral in Proposition \ref{prop. 2d NLY integral}. 
	\end{thm}
	\begin{proof}
		Let  $\epsilon<\gamma$ be small and $\tau \in [0, T]^2$ be two parameters to be chosen appropriately later, and let $\cC_\tau^{\gamma-\epsilon}(\xi)$ be a collection of paths $z$ in  $C^{\gamma-\epsilon}([0,\tau]^2;\RR^d)$ with the property that $z=\xi$ on the boundary (i.e. on $\partial[0,\tau]^2:=\{0\}\times [0,\tau] \cup [0,\tau]\times \{0\}$). In this space, the norms are restricted to the interval $[0,\tau]$, in the sense that $[\cdot]_{\gamma-\epsilon}:=[\cdot]_{{\gamma-\epsilon};[0,\tau]}$.  
		Note that this implies that all $z\in \cC^{\gamma-\epsilon}_\tau(\xi)$ can be decomposed as $z=\xi+y$, where $y$ is zero on the boundary, and by Remark \ref{rem: decomp of 2d functions} we have  that 
		\begin{equation*}
			\|z\|_{{\gamma-\epsilon}}\sim_\tau |z_0|+[z]_{\gamma-\epsilon} =|\xi_0|+[\xi]_{(1,0),{\gamma-\epsilon}}+[\xi]_{(0,1),{\gamma-\epsilon}}+[y]_{(1,1),{\gamma-\epsilon}}.  
		\end{equation*}
		Thus with the metric defined by $\cC^{\gamma-\epsilon}_\tau(\xi) \ni(x,y)\mapsto [x-y]_{(1,1),{\gamma-\epsilon}} $ it follows that $\cC^{\gamma-\epsilon}_\tau(\xi)$ is a complete affine metric space. Since $(1+\eta)\gamma>1$ and $\gamma>\frac{1}{2}$ choose  $\epsilon>0$ small such that $(1+\eta)({\gamma-\epsilon})>1$.  
		On the space $\cC^{\gamma-\epsilon}_\tau(\xi)$ define the solution mapping $M$ by 
		\begin{equation*}
			\cC^{\gamma-\epsilon}_\tau(\xi) \ni z \mapsto M_\tau (z):=\{\xi_t + \int_0^t A(\dd s,z_s)|\,t\in [0,\tau_1]\times [0,\tau_2] \}. 
		\end{equation*}
		Define now $\cB_\tau^{\gamma-\epsilon}(\xi)\subset \cC^{\gamma-\epsilon}_\tau(\xi)$ to be the closed unit ball in $\cC^{\gamma-\epsilon}_\tau(\xi)$ centered at $\xi$, i.e. 
		\begin{equation*}
			\cB_\tau^{\gamma-\epsilon}(\xi):=\{z\in \cC^{\gamma-\epsilon}_\tau(\xi)| \, [z-\xi]_{(1,1), {\gamma-\epsilon}}\leq 1\}. 
		\end{equation*}
		Note also that for any $z\in \cB^{\gamma-\epsilon}_\tau(\xi)$, then 
		\begin{equation}\label{eq:bound N xi}
			[z]_{\gamma-\epsilon}=[y]_{(1,1),{\gamma-\epsilon}}+[\xi]_{(1,0),{\gamma-\epsilon}}+[\xi]_{(0,1),{\gamma-\epsilon}}\leq 1+[\xi]_{(1,0),{\gamma-\epsilon}}+[\xi]_{(0,1),{\gamma-\epsilon}}=:N_\xi. 
		\end{equation}
		From here, a classical Picard fixed point argument can be developed; first we prove that the solution map  $M_\tau$ leaves $\cB^{\gamma-\epsilon}_\tau(\xi)$ invariant, i.e. $M_\tau(\cB^{\gamma-\epsilon}_\tau(\xi))\subset \cB^{\gamma-\epsilon}_\tau(\xi)$, and then we prove that the solution map is a contraction, i.e. for $z,y\in \cB^{\gamma-\epsilon}_\tau(\xi)$ then there exists a $q\in (0,1)$ such that 
		\begin{equation*}
			[M_\tau (y)-M_\tau(z)]_{\gamma-\epsilon} \leq q[y-z]_{\gamma-\epsilon}. 
		\end{equation*}
		We begin by proving invariance. For any $z\in \cB^{\gamma-\epsilon}_\tau(\xi)$ it follows from \eqref{eq:bound NLY integral}  that there exists a constant $C>0$ such that
		\begin{equation}\label{eq:local est}
			[M_\tau(z)-\xi]_{(1,1),{\gamma-\epsilon}}=[\int_0^\cdot A(\dd s,z_s)]_{(1,1),{\gamma-\epsilon}} \leq C m(\tau)^{\epsilon}\|A\|_{C^\gamma_tC^{1+\eta}_x}(1+[z]_{\gamma-\epsilon}^{1+\eta}),
		\end{equation}
		where we recall that  $m(\tau)^{\epsilon}=\tau_1^{\epsilon_1}\tau_2^{\epsilon_2}$. 
		Bounding the term $[z]_{\gamma-\epsilon}$ by $N_\xi$ as defined in \eqref{eq:bound N xi} we get
		\begin{equation*}
			[M_\tau(z)-\xi]_{(1,1),{\gamma-\epsilon}}\leq C m(\tau)^{\epsilon}\|A\|_{C^\gamma_tC^{1+\eta}_x}(1+N_\xi^{1+\eta} ), 
		\end{equation*}
		choosing $\tau^1$ small enough, we see that $M_{\tau^1}$ leaves the unit ball $\cB^{\gamma-\epsilon}_{\tau_1}(\xi)$ invariant. Note however that the choice $\tau^1$ depends on the boundary $\xi$, but only through the $(1,0)$ and $(0,1)$ H\"older norm, and does not depend on $z-\xi$,  which will be important later. 
		
		For contraction we invoke the stability inequality from Proposition \ref{prop:stability} to see that for $z,\tilde{z}\in \cB^{{\gamma-\epsilon}}_\tau(\xi)$ we have 
		\begin{equation*}
			[M_\tau(z)-M_\tau(\tilde{z})]_{(1,1),{\gamma-\epsilon}}\leq m(\tau)^{\epsilon}2K\|A\|_{C^\gamma_t C^{2+\eta}_x}(1+N_\xi)^{1+\eta}[z-\tilde{z}]_{(1,1),{\gamma-\epsilon}} ,
		\end{equation*}
		where we have used that $[z-\tilde{z}]_{{\gamma-\epsilon}}=[z-\tilde{z}]_{(1,1),{\gamma-\epsilon}}$ since $z$ and $\tilde{z}$ share the same boundary processes $\xi$. 
		Again, choose $\tau^2=\tau$ small, such that 
		\begin{equation*}
			m(\tau^2)^{\epsilon}2K\|A\|_{C^\gamma_tC^{2+\eta}_x}(1+N_\xi)^{1+\eta}<1, 
		\end{equation*}
		we see that $M_{\tau^2}$ is a contraction mapping from $\cB_{\tau^2}^{\gamma-\epsilon}(\xi)$ into $\cB_{\tau^2}^{\gamma-\epsilon}(\xi)$. Let $\bar{\tau}=\tau^1\wedge \tau^2$, and then it follows that there exists a unique  fixed point of $z\mapsto M_{\bar{\tau}}(z)$ on the interval $[0,\bar{\tau}]$.  
		For integers $k,j\geq 1$, the solution can now be iterated to rectangles on the form $[k\bar{\tau}_1,(k+1)\bar{\tau}_1]\times[j\bar{\tau}_2,(j+1)\bar{\tau}_2] \subset[0,T]^2$ by following the exact procedure provided in the proof of \cite[Theorem 25]{Harang2020}. The crucial point to obtain global existence is to use the fact $A$ and its derivatives is globally bounded and 
		the fact that the constant $N_\xi$, and thus $\tau^1$ and $\tau^2$  only depends on the boundary information, and not on $z-\xi$ for $z\in \cB_{\tau}^{\gamma-\epsilon}(\xi)$. The detailed proof of this point is rather lengthy and written in full detail in \cite{Harang2020}, and we therefore omit any further details here. 
		Once it is proven that the solution indeed exists on the full rectangle $[0,T]^2$, it follows that it is contained in $C^\gamma([0,T]^2;\RR^d)$. In fact, it is then readily seen that the solution $\theta$ to \eqref{eq:2d nly eq} is contained in $C^{\gamma-\epsilon}([0,T]^2;\RR^d)$. Indeed, observe that 
		\begin{equation*}
			|\square_{s,t}\theta|=|\int_s^t A(\dd s,\theta_s)|\lesssim_T \|A\|_{C^\gamma_tC^{1+\eta}_x}(1+[\theta]_{\gamma-\epsilon})m(t-s)^\gamma, 
		\end{equation*}
		and similar estimates shows that $[\theta]_{(1,0),\gamma}$ and $[\theta]_{(0,1),\gamma}$ are finite, and  thus we conclude that $\theta\in C^\gamma([0,T]^2;\RR^d)$.
	\end{proof}
	
	\begin{rem}
		In the above proof we have constructed a local solution on small rectangles $[0,\tau]=[0,\tau_1]\times [0,\tau_2]$, in order to iterate the solution to rectangles of the form $[k \tau_1, (k+1)\tau_1]\times [j \tau_2, (j+1) \tau_2]$, consistent with, and fully described in, \cite{Harang2020}. However, as pointed out by the anonymous referee,  based on the estimate in \eqref{eq:local est}, it is clear that one only really need to take $\tau=(\tau_1,T)$ for $\tau_1$ sufficiently small, and from there one may then do the iteration of the solution on small "strips" $[k \tau_1, (k+1)\tau_1]\times [0,T]$ of the rectangle $[0,T]^2$. Such a  solution method would certainly be similar to the one that is shown in \cite{Harang2020} and potentially reduce the length of some arguments slightly, but for brevity of the presentation we have here used the method outlined in \cite{Harang2020} to avoid a detail presentation of this step. 
	\end{rem}
	We conclude this section with the following proposition providing stability of the solutions to the 2D non-linear Young equations in terms of the non-linear function $A$. 
	\begin{prop}
		\label{stability of solutions}
		Consider two functions $A,\tilde{A}\in C^\gamma_tC^{2+\eta}_x$ for some $\gamma\in (\frac{1}{2},1]^2$ and $\eta \in (0,1)$ such that $(1+\eta)\gamma>1$.   Furthermore, let $\xi,\tilde{\xi}\in C^\gamma_t$ be such that for any $s<t\in [0,T]^2$,  $\square_{s,t} \xi=\square_{s,t} \tilde{\xi}=0$.    Let $\theta,\tilde{\theta}\in C^\gamma_t $ be two solutions to \eqref{eq:2d nly eq} driven by $(A,\xi)$ and $(\tilde{A},\xi)$ respectively, and assume there exists a constant  $M>0$ such that 
		\begin{equation}\label{eq:M const}
			\|\theta\|_\gamma \vee \|\tilde{\theta}\|_\gamma \vee \|A\|_{C^\gamma_tC^{2+\eta}_x}\vee \|\tilde{A}\|_{C^\gamma_tC^{2+\eta}_x}\leq M. 
		\end{equation}
		Then the solution map $(A,\xi)\mapsto \theta$ is continuous, and there exists a constant $C$ depending on $\gamma,M,T$ such that 
		\begin{equation}\label{eq:stability solution}
			[\theta-\tilde{\theta}]_{\gamma} \leq C( \|\xi-\tilde{\xi}\|_\gamma+ \|A-\tilde{A}\|_{C^\gamma_t C^{2+\eta}_x}).
		\end{equation}
	\end{prop}

	\begin{proof}
		
		First observe that the difference $\theta-\tilde{\theta}$ is given by 
		\begin{equation*}
			\theta_t-\tilde{\theta}_t=\xi_t-\tilde{\xi}_t+\int_0^t A(\dd s,\theta_s)-\int_0^t\tilde{A}(\dd s,\tilde{\theta}_s). 
		\end{equation*}
		Using the exact same techniques as for proving the stability result for the non-linear Young integral in \eqref{eq:bound diff NLY integrals}, we see that 
		\begin{multline*}
			|\int_s^t A(\dd r,\theta_r)-\int_s^t\tilde{A}(\dd r,\tilde{\theta}_r) - (\square_{s,t}A(\theta_s)-\square_{s,t}\tilde{A}(\tilde{\theta}_s))|
			\\
			\leq (C_1\|A-\tilde{A}\|_{C^\gamma_t C^{2+\eta}_x}+C_2(|\theta_s-\tilde{\theta_s}|+[\theta-\tilde{\theta}]_\beta))|t-s|^{\eta \gamma} m(t-s)^\gamma. 
		\end{multline*}
		Here $C_1$ and $C_2$ are given as in \eqref{eq:C const}, and by definition of $M$ we see that 
		\begin{equation*}
			C_1\vee C_2\leq K(M), 
		\end{equation*}
		for some monotone increasing function $K$. 
		Furthermore, using that $\square_{s,t}\xi=\square_{s,t}\tilde{\xi}=0$, we see from the same inequality that the following bound also holds
		\begin{multline*}
			|\square_{s,t}(\theta-\tilde{\theta})- (\square_{s,t}A(\theta_s)-\square_{s,t}\tilde{A}(\tilde{\theta}_s))|
			\\
			\leq K(M) (\|A-\tilde{A}\|_{C^\gamma_t C^{2+\eta}_x}+|\theta_s-\tilde{\theta_s}|+[\theta-\tilde{\theta}]_{\beta;[s,t]})|t-s|^{\eta \gamma} m(t-s)^\gamma.
		\end{multline*}
		For some $\tau\in (0,T)^2$ consider any interval $[\rho,\rho+\tau]$ such that $[\rho,\rho+\tau]\subset [0,T]^2.$ It follows that for $\beta<\gamma$ we have 
		\begin{equation*}
			[\theta-\tilde{\theta}-(\square_{\rho,\cdot }A(\theta_\rho)-\square_{\rho,\cdot}\tilde{A}(\tilde{\theta}_\rho))]_{(1,1),\gamma;[\rho,\rho+\tau]}\leq K(M) |\tau|^{\eta \gamma} (\|A-\tilde{A}\|_{C^\gamma_t C^{2+\eta}_x}+|\theta_\rho-\tilde{\theta}_\rho|+[\theta-\tilde{\theta}]_{\gamma;[\rho,\rho+\tau]}). 
		\end{equation*}
		By similar computations as above one can also show that 
		\begin{align*}
			[\theta-\tilde{\theta}-(\square_{\rho,\cdot }A(\theta_\rho)-\square_{\rho,\cdot}\tilde{A}(\tilde{\theta}_\rho))]_{(1,0),\gamma;[\rho,\rho+\tau]} &\lesssim_T  \tau_1^{\gamma \eta }K(M) (\|A-\tilde{A}\|_{C^\gamma_t C^{2+\eta}_x}+|\theta_{\rho}-\tilde{\theta}_{\rho}|+[\theta-\tilde{\theta}]_{\gamma;[\rho,\rho+\tau]}),
			\\
			[\theta-\tilde{\theta}-(\square_{\rho,\cdot }A(\theta_\rho)-\square_{\rho,\cdot}\tilde{A}(\tilde{\theta}_\rho))]_{(0,1),\gamma;[\rho,\rho+\tau]} &\lesssim_T\tau_2^{\gamma \eta}K(M) (\|A-\tilde{A}\|_{C^\gamma_t C^{2+\eta}_x}+|\theta_{\rho}-\tilde{\theta}_{\rho}|+[\theta-\tilde{\theta}]_{\gamma;[\rho,\rho+\tau]}). 
		\end{align*}
		Combining these estimates, it follows that there exists a constant $C=C(T,K(M),\gamma \eta)>0$ such that 
		\begin{equation*}
			[\theta-\tilde{\theta}-(\square_{\rho,\cdot }A(\theta_\rho)-\square_{\rho,\cdot}\tilde{A}(\tilde{\theta}_\rho))]_{\gamma;[\rho,\rho+\tau]} \leq C\left(\|A-\tilde{A}\|_{C^\gamma_t C^{2+\eta}_x} +  |\theta_\rho -\tilde{\theta}_{\rho}|+|\tau|^{\gamma \eta} [\theta-\tilde{\theta}]_{\beta;[\rho,\rho+\tau]}\right).
		\end{equation*}
		In particular, choosing $\tau$ small enough such that  
		\begin{equation*}
			|\tau|^{\eta \gamma} \leq  \frac{1}{2C}, 
		\end{equation*} 
		it follows that 
		\begin{equation*}
			[\theta-\tilde{\theta}-(\square_{\rho,\cdot }A(\theta_\rho)-\square_{\rho,\cdot}\tilde{A}(\tilde{\theta}_\rho))]_{\gamma;[\rho,\rho+\tau]} \leq 2 C( |\theta_{\rho}-\tilde{\theta}_{\rho}| + \|A-\tilde{A}\|_{C^\gamma_t C^{2+\eta}_x}). 
		\end{equation*}
		and in particular, reformulating using the triangle inequality, we see that 
		\begin{equation}
			[\theta-\tilde{\theta}]_{\gamma;[\rho,\rho+\tau]} \leq 2 C( |\theta_{\rho}-\tilde{\theta}_{\rho}| + \|A-\tilde{A}\|_{C^\gamma_t C^{2+\eta}_x}).
		\end{equation}
		Note that this inequality holds for any $\rho\in [0,T)^2$ such that $[\rho,\rho+\tau]\subset [0,T]^2$, an in particular for the rectangle $[0,\tau]$. Iterating the inequality obtained on this interval  to any interval $[k\tau, (k+1)\tau]\subset [0,T]^2$, using that the relation that $|x_t|\leq |x_0|+[x]_{\gamma}$, one can show that on any interval $[\rho,\rho+\tau]$
		\begin{equation*}
			[\theta-\tilde{\theta}]_{\gamma;[\rho,\rho+\tau]} \lesssim  \|\xi-\tilde{\xi}\|_\gamma + \|A-\tilde{A}\|_{C^\gamma_t C^{2+\eta}_x}. 
		\end{equation*}
		We therefore conclude by the 2D-H\"older norm scaling property proven in \cite[Proposition 24]{Harang2020}, that \eqref{eq:stability solution} holds. 
	\end{proof}

	\section{Regularization of SDEs on the plane.}\label{sec:regularization SDE on plane}
	Consider the stochastic differential equation formally given by 
	\begin{equation}\label{eq:integral eq}
		x_t =\xi_t +\int_0^t b(x_s)\dd s + w_t, \quad t\in [0,T]^2, 
	\end{equation}
	where $\xi_t=\xi^1_{t_1}1_{t_2=0}+\xi_{t_2}^21_{t_1=0}$ is a function supported on the boundary $\{0\}\times [0,T]\cup [0,T]\times \{0\}$. The integral equation can  therefore be seen to be equipped with the two  boundary conditions $x_{t_1,0}=\xi^1_{t_1}+w_{t_1, 0} $ and  $x_{0,t_2}=\xi^2_{t_2}+w_{0, t_2}$. 
	The goal of this section is to prove well-posedness of this equation, even in the case when $b$ is distributional (in the sense of generalized functions) given that $w$ provides a sufficiently regularizing effect. 
	\begin{rem}
		\label{discussion local time}
		In this article the formulation of the results for regularization by noise will be in terms of the potential  regularity of the local time, similar to the approach in  \cite{Harang2021RegularityOL,HarangPerkowski2020}. Similar type of pathwise regularization by noise results can also be formulated as requirements on the so called averaged field associated with $w$ and the drift $b$, as in \cite{Catellier2016}. On two dimensional domains, the local time and occupation measure related to stochastic fields is a well studied topic, and some types of regularity estimates are well known, see e.g. \cite[Theorem 28.1]{GemanHorowitz1980} where it is established that the $(H_1, \dots, H_N)$-fractional Brownian sheet on $[0,T]^N$ with $H_1=\dots=H_N=H$ has a local time $L^{H}_t$ contained $H^{\frac{N}{2H}-\frac{d}{2}}$ $\PP$-almost surely (see also \cite{ZhangXiao2002}, \cite{Ayache2008} for joint space-time continuity of the local time in this setting). However, such regularity estimates are not sufficient in order to apply non-linear Young theory, as one still misses proofs of higher quantified joint space time regularity of these occupations measures (in particular, estimates that provide at the same time  a quantified Young regularity in time and higher spatial regularity appear lacking. While \cite{ZhangXiao2002}, \cite{Ayache2008} do provide some Hölder regularity in time for fixed space points respectively Hölder regularity in space for fixed time points, estimates that establish jointly Hölder continuity in time on a Bessel potential scale of high order as in space as in Theorem \ref{regularity fBS} appear to be new). 
	\end{rem}

	Suppose now that  $w:[0,T]^2\rightarrow \RR^d$ is a stochastic field. Let $\mu^w$ denote the occupation measure of $w$, defined as follows for a Borel set $A\in \cB(\RR^d)$,
	\begin{equation*}
		\mu^w_t(A) = \lambda\{s<t\in [0,T]^2|\, w_s\in A\},
	\end{equation*}
	where $\lambda$ is the Lebesgue measure. 
	If the occupation measure is absolutely continuous with respect to the Lebesgue measure on $\mathbb{R}^d$ it admits a density, $L^w$, i.e. 
	\begin{equation*}
		\mu^w_t(A)=\int_A L^w_t(z)\dd z,\quad t\in [0,T]^2. 
	\end{equation*}
	The function $L^w:[0,T]^2 \times \RR^d \rightarrow \RR_+$ is called the local time associated to $w$. 
	Given a bounded measurable function  $b:\RR^d\rightarrow \RR^d$ and $x\in \RR^d$,  the following local time formula holds 
	\begin{equation}\label{eq:local time formula}
		\int_s^t b(x+w_r)\dd r= \int_{\RR^d}b(x-z)\square_{s,t} L^{-w}(z)\dd z = (b\ast \square_{s,t}L^{-w})(x). 
	\end{equation}
	
	\begin{rem}
		For the reader familiar with the concept of averaged fields in the pathwise regularization by noise approach developed by Catellier, Galeati and Gubinelli in \cite{Catellier2016,galeati2020noiseless,galeati2020prevalence}, the left hand side of the above equation could be seen as a two dimensional extension of the averaged field associated to $b$ and $w$. That is, 
		\begin{equation*}
			T^{w}b(t,x)=\int_0^t b(x+w_r)\dd r,\quad t\in [0,T]^2,
		\end{equation*}
		where we stress that the integral above is a double integral, and from convention $\dd r=\dd r_2 \dd r_1$. 
	\end{rem}

	With the aim of proving well-posedness of \eqref{eq:integral eq} in the case when $b$ is truly distributional, we need to make sense of the integral appearing in \eqref{eq:integral eq}. Similarly to the (1D) pathwise regularization approach, we will construct this integral in terms of a non-linear Young integral. That is, let us first reformulate \eqref{eq:integral eq} by setting $\theta=x-w$. Then  formally $\theta$ solves the equation  
	\begin{equation}\label{eq:transformed}
		\theta_t=\xi_t+\int_0^t b(\theta_s+w_s)\dd s. 
	\end{equation}
	In the case when $b$ is continuous the integral appearing above can be constructed in the Riemann sense; for a sequence of partitions of $\{\cP^n\}$ of  $[0,t_1]\times [0,t_2]$ (constructed as in Definition \ref{def: partition}) with mesh going to zero when $n$ tends to infinity, we have that 
	\begin{equation}\label{eq:Riemann}
		\int_0^t b(\theta_s+w_s)\dd s=\lim_{n\rightarrow\infty}\sum_{[u_1,v_1]\times [u_2,v_2]\in \cP^n} b(\theta_u+w_u)\square_{u,v} id,  
	\end{equation}
	where $\square_{u,v} id :=\int_{u_1}^{v_1}\int_{u_2}^{v_2}\dd r_2\dd r_1$. 
	An alternative approach to constructing this integral is in terms of the non-linear Young integral. Suppose the local time associated with $w$ is differentiable in its spatial variable and (2D) $\gamma$-H\"older continuous in the time variable with $\gamma>\frac{1}{2}$. Then we can use Proposition \ref{prop. 2d NLY integral} to show that the following integral exists: 
	\begin{equation}\label{eq:conv NLY int}
		\int_0^t (b\ast L^{-w})(\dd s, \theta_s):=\lim_{n\rightarrow\infty}\sum_{[u_1,v_1]\times [u_2,v_2]\in \cP^n} (b\ast \square_{u,v} L^{-w})(\theta_u). 
	\end{equation}
	In the next proposition we prove that the above integral indeed exists in the non-linear Young sense, and that it agrees with the classical Riemann integral in the case of continuous functions $b$. 
	
	\begin{prop}\label{prop:conv NLY integral}
		Consider $p,q\in [1,\infty]$ such that $\frac{1}{p}+\frac{1}{q}=1$.  Let $b\in B^\zeta_{p,p}(\RR^d)$ for some $\zeta\in \RR$ and assume $w:[0, T]^2\to \RR^d$ to be continuous such that $L^{-w}\in C^\gamma_tB^\kappa_{q,q}(\RR^d)$, with $\gamma\in (\frac{1}{2},1]^2$ and $\zeta+\kappa>1+\eta$ for some $\eta\in (0,1)$. Suppose $\theta\in C^\alpha_t$ for $\alpha \in (0,1)^2$ such that $\alpha \eta+\gamma>1$.  Then  the nonlinear Young integral defined in \eqref{eq:conv NLY int} exists. Furthermore, if $b$ is continuous, then this integral agrees with the classical Riemann integral.  
	\end{prop}
	\begin{proof}
		Set $\square_{s,t}A(x):=(b\ast \square_{s,t}L^{-w})(x)=\int_s^tb(x+w_r)\dd r$. By Young's convolution inequality (see e.g. \cite{BahCheDan}), it follows that 
		\begin{equation}\label{eq:Young conv ineq}
			\|\square_{s,t}A\|_{C^{\kappa+\zeta}_x}\lesssim \|b\|_{B^\zeta_{p,p}(\RR^d)}\|L^{-w}\|_{C^\gamma_tB^\kappa_{q,q}(\RR^d)}m(t-s)^\gamma. 
		\end{equation}
		This implies that $A\in C^\gamma_t C^{\kappa+\zeta}_x$. Furthermore, $t\mapsto \square_{0,t}L^w$ is clearly $0$ on the boundary, as illustrated in Remark \ref{rem: decomp of 2d functions}. Thus,  since $\kappa+\zeta>1+\eta$ and $\theta\in C^\alpha_t$ with $\eta\alpha +\gamma>1$ it then follows from Proposition \ref{prop. 2d NLY integral} that the integral 
		\begin{equation*}
			\int_0^t (b\ast L^{-w})(\dd s, \theta_s)=\int_0^t A(\dd s,\theta_s),
		\end{equation*}
		exists in the sense of \eqref{eq:2dnly}.

		Suppose now that $b$ is continuous so that the Riemann integral in \eqref{eq:Riemann} exists. Remark that then $A_t(x)=\int_0^t b(x+w_r)\dd r$, i.e. $\partial_{t_1}\partial_{t_2}A_s(x)=b(x+w_s)$ is continuous. Hence, by Lemma \ref{consistency 2d nonlinear young}, the 2d-nonlinear Young integral constructed coincides with the corresponding Riemann integral.
	\end{proof}

	Now that the non-linear Young integral for the convolution with a function $b$ and the local time is well defined, we will move on to prove existence and uniqueness of solutions to \eqref{eq:integral eq}. 
	However, as we are interested in allowing for distributional $b$, we need a rigorous concept of solution which behaves well under approximation. This is provided in the following definition. 
	
	\begin{defn}\label{def: concept of solution}
		Let $w:[0,T]^2\rightarrow \RR^d$  be a continuous field, and $b\in \cS(\RR^d)'$ (the space of Schwartz distributions). Assume that $b\ast L^{-w}\in C^{\gamma}_tC^{2+\eta}_x$ for some $\gamma\in (\frac{1}{2},1]^2$ and for some $\eta\in (0,1)$ such that   $(1+\eta)\gamma>1$. We say that $x\in C([0,T]^2; \RR^d)$ is a solution to \eqref{eq:integral eq} if there exists a $\theta\in C^\gamma([0,T]^2;\RR^d)$, such that $x=w+\theta$, and $\theta$  satisfies 
		\begin{equation}\label{eq:NLY regularized}
			\theta_t=\xi_t+\int_0^t (b\ast L^{-w})(\dd s,\theta_s), \quad t\in [0,T].
		\end{equation}
		Here, $\xi_t=\xi^1_{t_1}1_{t_2=0}+\xi_{t_2}^21_{t_1=0}$ and $\xi\in C^\beta_t$ for some $\beta\geq \gamma$, the integral is understood as a non-linear Young integral as defined in Proposition \ref{prop:conv NLY integral} and the equation is interpreted in the non-linear Young sense (see Theorem \ref{thm:nly equation}). We call $\xi$ boundary data of $\theta$. The boundary data of $x$ is accordingly $\xi+w$. 
	\end{defn}
	The next result provides simple conditions for the existence and uniqueness of solutions in terms of the regularity of the (possibly distributional) coefficient $b$ and the regularity of the local time associated to the continuous field $w$. 
	
	\begin{thm}\label{thm:existence and uniqueness}
		Let $p,q\in [1,\infty]$ be such that $\frac{1}{p}+\frac{1}{q}=1$. Assume $b\in B^\zeta_{p,p}(\RR^d)$ for $\zeta \in \RR$ and that $w\in C([0,T]^2;\RR^d)$ has an associated local time $L^{-w}\in C^\gamma_t B^\alpha_{q,q}(\RR^d)$ for some $\gamma\in (\frac{1}{2},1]^2$ and $\alpha \in (0,1)^2$. If  there exists an   $\eta\in (0,1)$  such that 
		\begin{equation}\label{eq:conditions for existence.}
			\gamma(1+\eta)>1\quad  and \quad \zeta+\alpha>2+\eta, 
		\end{equation}
		then there exists a unique solution $x\in C([0,T]^2;\RR^d)$ to equation \ref{eq:integral eq}, where the solution is given in the sense of Definition  \ref{def: concept of solution}. 
	\end{thm}
	
	\begin{proof}
		We know that under the conditions in \eqref{eq:conditions for existence.} the non-linear Young integral in \eqref{eq:NLY regularized} is well defined according to Proposition \ref{prop:conv NLY integral}, and it follows by Theorem \ref{thm:nly equation} that a unique solution to \eqref{eq:NLY regularized} exists in $C^\gamma_t$. Setting $x=w+\theta$ it follows that $x\in C([0,T]^2;\RR^d)$, and that $x$ is a solution in the sense of Definition \ref{def: concept of solution}. 
	\end{proof}
	\begin{rem}
		Note that in the case when $b$ is a continuous function and the conditions of Theorem \ref{thm:existence and uniqueness} is satisfied, then the solution coincides with the classical one. Indeed, in this case, the simple  transformation \eqref{eq:integral eq} given in \eqref{eq:transformed} holds, and since the non-linear Young integral agrees with the Riemann integral, in this case, the two concepts of solution also agree. 
	\end{rem}
	
	With the stability result from Proposition \ref{stability of solutions} it also follows that smooth approximations of a solution converge to the solution of the non-linear Young equation. 
	\begin{cor}\label{cor. stability}
		Suppose $b$, $w$ and $L^{-w}$ satisfies the conditions of Theorem \ref{thm:existence and uniqueness} such that a unique solution $x$  to \eqref{eq:integral eq} with boundary data $\xi+w$ exists in the sense of Definition \ref{def: concept of solution}. Let $\{b^n\}_{n\in \NN}$ be a sequence of smooth functions approximating $b$ such that $\lim_{n\rightarrow \infty} \|b^n-b\|_{B^\zeta_{p,p}(\RR^d)}=0$. Denote by $\{x^n\}_{n \in \NN}\subset C([0,T]^2;\RR^d)$ the sequence of solutions with boundary data $\xi+w$ constructed from the sequence of  solutions to \eqref{eq:integral eq} where the drift of $x^n$ is given by $b^n$. Then $x^n\rightarrow x$ in $C([0,T]^2;\RR^d)$, and we have that for any $\beta\in (0,\gamma)$ and any $n\in \NN$
		\begin{equation*}
			\|x^n-x\|_{\infty} \leq C\|b^n-b\|_{B^\zeta_{p,p}(\RR^d)}. 
		\end{equation*}
	\end{cor}
	\begin{proof}
		This follows directly from Proposition \ref{stability of solutions}, the fact that $\|x^n-x\|_{\infty}\leq |x_0^n-x_0|+ [x^n-x]_\beta$, together with the interpretation of the solutions in the non-linear Young sense, as illustrated in Proposition \ref{prop:conv NLY integral} and Theorem \ref{thm:existence and uniqueness}. 
	\end{proof}

	\subsection{Regularity of the local time of the fractional Brownian sheet}
	While the above theorem provides explicit conditions for the existence and uniqueness of solutions to the equation in terms of regularity of the local time associated with the continuous field $w$, it does not provide any further conditions on the field $w$ to guarantee that the local time indeed has the assumed regularity. The above theorem is therefore abstract in itself, and one needs to study space-time regularity properties of local times associated with various continuous (stochastic) fields in order to get concrete conditions on the field $w$ and $b$ to guarantee existence and uniqueness. In the following , we derive joint space-time regularity estimates for the local time associated with two types of noises: the fractional Brownian sheet and sums of independent fractional Brownian motions in distinct variables. The following theorem can be considered an extension of \cite[Theorem 3.1]{HarangPerkowski2020} to the two dimensional setting.
	\begin{thm}
		\label{regularity fBS}
		Let $w:[0,T]^2\to \RR^d$ be a fractional Brownian sheet of Hurst parameter $H=(H_1, H_2)$ on $(\Omega, \mathcal{F}, \mathbb{P})$. Suppose that 
		\[
		\lambda<\frac{1}{2(H_1\vee H_2)}-\frac{d}{2}
		\]
		Then for almost every $\omega\in \Omega$, $w$ admits a local time $L$ such that for $\gamma_1\in (1/2, 1-(\lambda+\frac{d}{2})H_1)$ and $\gamma_2\in (1/2, 1-(\lambda+\frac{d}{2})H_2)$ 
		\[
		\norm{\square_{s,t} L}_{H^\lambda_x}\lesssim (t_1-s_1)^{\gamma_1}(t_2-s_2)^{\gamma_2}
		\]
	\end{thm}
	\begin{proof}
		We proceed similar to \cite[Theorem 3.1]{HarangPerkowski2020}. Recall that if $\mu$ denotes the occupation measure associated with $w$, we have by the occupation times formula
		\[
		\square_{s,t}\hat{\mu}(z)=\int_s^t e^{izw_r}dr.
		\]
		Recall also that it was shown in \cite[Theorem 3.1]{HarangPerkowski2020} that for a $d$-dimensional fractional Brownian motion $B^{H_1}$ of Hurst parameter $H_1$, one obtains the bound
		\begin{align} 
			\norm{\int_{s_1}^{t_1} e^{i\alpha z\cdot B^{H_1}_r}dr}_{L^m(\Omega)}\lesssim (1+|z|^2)^{-\lambda'/2}|\alpha|^{-\lambda'}(t_1-s_1)^{1-\lambda'H_1}
		\end{align}
		provided $1-\lambda'H_1>1/2$ thanks to the stochastic sewing lemma \cite{le2018}. Moreover, remark that for the Brownian sheet $w_r=w_{r_1, r_2}$, we have that for fixed time points $r_2\in [0,T]$, $w_{r_1, r_2}\simeq r_2^{H_2}B_{r_1}^{H_1}$ in law, where $B^{H_1}$ is a standard fractional Brownian motion of Hurst parameter $H_1$. This implies that 
		\begin{align*}
			\norm{\int_s^te^{izw_r}dr}_{L^m}&\leq \int_{s_2}^{t_2}\norm{\int_{s_1}^{t_1}e^{izw_r}dr_1}_{L^m}dr_2\\
			&=\int_{s_2}^{t_2}\norm{\int_{s_1}^{t_1}e^{izr_2^{H_2}B^{H_1}_{r_1}}dr_1}_{L^m}dr_2\\
			&\lesssim\int_{s_2}^{t_2} (1+|z|^2)^{-\lambda'/2}r_2^{-H_2\lambda'}(t_1-s_1)^{1-\lambda'H_1}dr_2\\
			&\lesssim (1+|z|^2)^{-\lambda'/2}(t_1-s_1)^{1-\lambda'H_1}(t_2-s_2)^{1-\lambda'H_2}
		\end{align*}
		where we need to require $1-\lambda'H_2>1/2$ as to assure $\gamma_2>1/2$. By definition of Bessel-potential spaces, we have therefore obtain by Minkowski's integral inequality 
		\begin{align*}
			\mathbb{E}[\norm{\square_{s,t}\mu}_{H^s}^p]^{1/p}&=\left(\mathbb{E}\left(\int_{\RR^d}\square_{s,t}\hat{\mu}(z)^2(1+|z|^2)^{s}dz \right)^{p/2}\right)^{1/p}\\
			&\leq \left(\int_{\RR^d}\norm{|\square_{s,t}\hat{\mu}(z)|^2}_{L^{p/2}(\Omega)}(1+|z|^2)^{s}dz \right)^{1/2}\\
			&=\left(\int_{\RR^d}\norm{|\square_{s,t}\hat{\mu}(z)}_{L^{p}(\Omega)}^{2}(1+|z|^2)^{s}dz \right)^{1/2}\\
			&\lesssim (t_1-s_1)^{1-\lambda'H_1}(t_2-s_2)^{1-\lambda'H_2}\int_{\RR^d}(1+|z|^2)^{s-\lambda'}dz\\
			&\lesssim (t_1-s_1)^{1-\lambda'H_1}(t_2-s_2)^{1-\lambda'H_2}\int_{0}^\infty(1+|r|^2)^{s-\lambda'}r^{d-1}dr\\
			&\lesssim (t_1-s_1)^{p(1-\lambda'H_1)}(t_2-s_2)^{p(1-\lambda'H_2)}
		\end{align*}
		provided $s<\lambda'-d/2$. We can then conclude by the joint Kolmogorov continuity theorem as expressed in \cite[Theorem 3.1]{HU20133359} to obtain for $\gamma_1<1-\lambda'H_1$ and $\gamma_2<1-\lambda'H_2$
		\[
		\norm{\square_{s,t}\mu}_{H^s}\lesssim (t_1-s_1)^{\gamma_1}(t_2-s_2)^{\gamma_2}
		\]
		It thus follows that for such $\gamma=(\gamma_1,  \gamma_2)$, we may conclude 
		\[
		L\in C^\gamma_tH^s
		\]
		where
		\[
		s<\lambda'-\frac{d}{2}<\frac{1}{2(H_1\vee H_2)}-d/2.
		\]
		
	\end{proof}
	\begin{rem}
		Let us remark that we expect the above Theorem \ref{regularity fBS} to be far from optimal: In particular, no genuinely two dimensional stochastic cancellations have been employed, meaning regularization is not obtained from "both directions", but rather limited by the one with the biggest Hurst parameter i.e. $H_1\vee H_2$. Indeed, note that already the above proof is exploiting self-similarity properties to transfer the one parameter setting to the present two parameter setting. As already in \cite{HarangPerkowski2020}, a crucial role in the regularity estimates for local times was played by the stochastic sewing Lemma, we expect that in our setting, a "2D stochastic sewing Lemma" not yet available in the literature (see however \cite{kern} for a stochastic reconstruction theorem very close in spirit) might prove instrumental in establishing regularization from "both directions". 
	\end{rem}
	Combing Theorem \ref{thm:existence and uniqueness} with the above \ref{regularity fBS}, we immediately obtain:
	\begin{thm}
		Let $w:[0,T]^2\to \RR^d$ be a fractional Brownian sheet of Hurst parameter $H=(H_1, H_2)$ on $(\Omega, \mathcal{F}, \mathbb{P})$. Let $b\in H^\zeta$ for $\zeta\in \RR$. Suppose that 
		\begin{equation}
			\zeta>3-\frac{1}{2(H_1\vee H_2)}+\frac{d}{2},
		\end{equation}
		then for almost all $\omega\in \Omega$ independent of $b$ there exists a unique solution $x\in C([0,T]^2;\RR^d)$ to equation \ref{eq:integral eq}, where the solution is given in the sense of Definition  \ref{def: concept of solution}
	\end{thm}

	\subsection{Regularity of the local time of the sum of two one-parameter fractional Brownian motions}
	\begin{lem}\label{lem:fbmconv}
		Let $\beta^1$ and $\beta^2$ are two fractional Brownian motions on a probability space $(\Omega,\cF,\PP)$ with Hurst parameters $H_1,H_2 \in (0,1)$. Then, for almost all $\omega\in \Omega$, the local time $L^{-w}$ associated with $w=\beta^1+\beta^2$ is given by the convolution of local times $L^{-\beta^1}$ and $L^{-\beta^2}$, i.e.
		\begin{equation*}
			L^{-w}_t(x)=L_{t_1}^{-\beta^1}\ast L^{-\beta^2}_{t_2} (x), \quad t=(t_1,t_2) \in [0,T]^2.
		\end{equation*}
	\end{lem}
	\begin{proof}
		Let $b: \RR^d \to \RR$ be any measurable function and fix any $\omega \in \Omega$. To keep the notation simple we avoid writing $\omega$ explicitly.  By applying the local time formula twice we have
		\begin{align*}
			b\ast L^{-w}_t(x)&=\int_0^t b(x+w_r)\dd r
			=\int_{0}^{t_1}\int_0^{t_2} b(x+\beta_{r_1}^1+\beta_{r_2}^2)\dd r_1\dd r_2
			\\
			& = \int_0^{t_1}\int_{\RR^d} b(x+\beta^1_{r^1}-z)L_{t_2}^{-\beta^2}(z)\dd z\dd r_1 
			\\
			&=\int_{\RR^d}\int_{\RR^d} b(x-z'-z)L_{t_2}^{-\beta^2}(z)L_{t_1}^{-\beta^1}(z')\dd z\dd z'
			\\
			&= b\ast L_{t_1}^{-\beta^1}\ast L^{-\beta^2}_{t_2} (x). 
		\end{align*}
		
		Hence the result. 
	\end{proof}

	\begin{rem}
		Observe that the result proved in Lemma \ref{lem:fbmconv} is pathwise and thus holds for any
		random sampling of $\beta^1$ and $\beta^2$, regardless of whether they are independent or not.
	\end{rem}
	
	The next result is an interesting application of Theorem \ref{thm:existence and uniqueness} about the regularization by a special type of two-dimensional stochastic field which is the sum of two fBms.

	\begin{thm}\label{thm:fbm reg}
		Let $w_t :=\beta^1_{t_1}+\beta^2_{t_2}$, where $\beta^1$ and $\beta^2$ are two  fractional Brownian motions  on a probability space $(\Omega,\cF,\PP)$ with the Hurst parameters, respectively, $H_1,H_2 \in (0,1)$.
		Then, if $b\in B^\zeta_{1,1}(\RR^d)$, where $\zeta$ satisfies
		\begin{equation}\label{eq:fbm assumption}
			\zeta >d+3 - \frac{1}{2H_1}- \frac{1}{2H_2},
		\end{equation}
		then, for almost all $\omega\in \Omega$, there exists a unique solution to the equation 
		\begin{equation*}
			x_t(\omega) = \int_0^t b(x_s(\omega))\dd s +w_t(\omega), \quad t\in [0,T]^2, 
		\end{equation*}
		where the solution is given in the sense of Definition \ref{def: concept of solution} with $\xi_t=0$. 
	\end{thm}
	\begin{proof}
		We know that $\PP$-a.s. the local times $L^{-\beta^1}$ and $L^{-\beta^2}$ associated to $\beta^1$ and $\beta^2$ are contained in $C^{\gamma_1}_tH^{\frac{1}{2H_1}-\frac{d}{2}-\epsilon}_x$ and $C^{\gamma_2}_tH^{\frac{1}{2H_2}-\frac{d}{2}-\epsilon}_x$ for any $\epsilon>0$ and some $\gamma_1,\gamma_2>\frac{1}{2}$, see e.g. \cite{HarangPerkowski2020}. Let us set $\gamma=(\gamma_1,\gamma_2)$ and choose $\eta\in (0,1)$ such that $(1+\eta)\gamma>1$.

		Since, from Lemma \ref{lem:fbmconv},  the local time $L^{-w}$ is given by a convolution of two (one-dimensional) local times $L^{-\beta^1}\ast L^{-\beta^2}$, its regularity is found from Young's convolution inequality in Besov spaces (see e.g. \cite{BahCheDan}). Thus its regularity in the spatial variable is given as the sum of the spatial  regularities of the one dimensional local times, i.e. for all $t\in [0,T]^2$,  $L^{-w}_t\in C_x^{\frac{1}{2H_1}+\frac{1}{2H_2}-d}\simeq B^{\frac{1}{2H_1}+\frac{1}{2H_2}-d}_{\infty,\infty}$. 
		
		Furthermore, it is readily checked that 
		\begin{equation*}
			\square_{s,t}L^{-w} = (L_{t_1}^{-\beta^1}-L_{s_1}^{-\beta^1})\ast(L_{t_2}^{-\beta^2}-L_{s_2}^{-\beta^2}), 
		\end{equation*}
		and thus it follows by elementary computations that $t\mapsto L^{-w}_t\in C^\gamma_t$, and we conclude that $L^{-w}\in C^\gamma_tC^{\frac{1}{2H_1}+\frac{1}{2H_2}-d}_x$. Again by invoking the Young's convolution inequality  using that $b\in B^{\zeta}_{1,1}$ it follows by the same estimate as in \eqref{eq:Young conv ineq} together with the, we get that $b\ast L^{-w}\in C^\gamma_tC^{\zeta+\frac{1}{2H_1}+\frac{1}{2H_2}-d}_x$. Now it follows from the assumption \eqref{eq:fbm assumption} that we can apply Theorem \ref{thm:existence and uniqueness}, which concludes the proof. 
	\end{proof}

	Our next result shows that it is sufficient to have only one  of $\beta_1$ and $\beta_2$ random in Theorem \ref{thm:fbm reg}.   
	
	\begin{lem}\label{lem:fbm reg-1}
		Let $[0,T]^2$, $\xi_t$ and $b$ as in Theorem \ref{thm:fbm reg}. Let $w_t :=\beta_{t_1} + f(t_2), t=(t_1,t_2) \in [0,T]^2$, where $\beta$ is a fractional Brownian motion on a probability space $(\Omega,\cF,\PP)$ with the Hurst parameter $H \in (0,1)$ and  $f:[0,T] \to \RR^d$ is a measurable function. Then, if $\zeta$ satisfies the relation
		\begin{equation}\label{eq:fbm assumption-1}
			\zeta >d+3 - \frac{1}{2H},
		\end{equation}
		then, the conclusion of Theorem \ref{thm:fbm reg} holds true. 
	\end{lem}
	\begin{proof}
		First observe that by definition of occupation measure 
		$$ \| \mu^{f}_{s_2,t_2} \|_{TV} = |t_2-s_2|, \quad \forall s_2<t_2 \in [0,T],$$
		where $\mu^{f}$ is the occupation measure of $\beta^2$ and $\| \cdot \|_{TV}$ is the total variation norm. Next, since $\| \mu^{f}_{\cdot,\cdot} \|_{B_{1,\infty}^1} \lesssim \| \mu^{f}_{\cdot,\cdot} \|_{TV} $ and, as in Theorem \ref{thm:fbm reg},  $L^{-\beta^1} \in C^{\gamma}_t H^{\frac{1}{2H}-\frac{d}{2}-\epsilon}_x, \PP$-a.s., local time relation \eqref{eq:local time formula} and convolution inequalities give, for $s=(s_1,s_2) < t=(t_1,t_2) \in [0,T]^2$,  
		\begin{align}
			& \| b \ast \square_{s,t}L^{-w}\|_{C_x^{\zeta+\frac{1}{2H} -d -\epsilon }}  = \| b \ast L_{s_1,t_1}^{-\beta^1}\ast \mu^f_{s_2,t_2} \|_{C_x^{\zeta+\frac{1}{2H} -d -\epsilon }} \nonumber\\
			& \qquad \lesssim \|b\|_{B_{1,1}^{\zeta}} \|L_{s_1,t_1}^{-\beta^1} \|_{C_x^{\frac{1}{2H} -d -\epsilon }}  \|\mu^f_{s_2,t_2} \|_{TV} \lesssim |t_1-s_1|^\gamma |t_2-s_2|. \nonumber
		\end{align}
		This implies that $b \ast \square_{s,t}L^{-w} \in C_t^{(\gamma,1)} C_x^{\zeta+\frac{1}{2H} -d -\epsilon }$. Then, as in the proof of previous theorem, the conclusion follows by applying Theorem \ref{thm:existence and uniqueness}. 
	\end{proof}

	\section{Wave equation with noisy boundary}\label{sec:wave eqn}
	
	\subsection{Statement of the problem}
	In the following section, we show how the theory of 2D non-linear Young equations can be employed in the study of Goursat boundary regularization for wave equations with singular non-linearities. More precisely, we intend to study the problem
	\begin{equation}\label{eq: wave eq, w/noisy boundary}
		\begin{split}
			\left( \frac{\partial^2}{\partial x^2}- \frac{\partial^2}{\partial y^2}\right)u=h(u(x,y)),
		\end{split}
	\end{equation}
	on $(x, y)\in R_{\frac{\pi}{4}}\circ [0,T]^2$ (here $R_{\frac{\pi}{4}}$ denotes the rotation operator of the plane by $\pi/4$)  subject to the boundary conditions along characteristics
	\begin{equation}
		\begin{split}
			u(x,y)&=\beta^1(y) \qquad \text{if}\ y=x\\
			u(x,y)&=\beta^2(y) \qquad \text{if}\ y=-x,
		\end{split}
		\label{boundary wave eqn}
	\end{equation}
	and the consistency condition $\beta^1(0)=\beta^2(0)$ for a potentially distributional non-linearity $h$. Note that for distributional $h$, it is a priori even unclear what is meant by a solution to \eqref{eq: wave eq, w/noisy boundary}.
	We therefore start by considering smooth mollifications $h^\epsilon=h*\rho^\epsilon$ as non-linearities, for which a change of coordinates yields a reformulation of \eqref{eq: wave eq, w/noisy boundary} as a Goursat problem \eqref{eq:Goursat type noisy boundary} which in turn can be analysed as a 2D non-linear Young equation \eqref{problem to solve}. As seen in the previous section, 2D non-linear Young equations can be well posed even in the case of distributional $h$, provided sufficient regularization by the boundary conditions $\beta^1, \beta^2$ is assumed. Moreover, they enjoy the stability property of Proposition \ref{stability of solutions}. In particular, this will imply that the sequence of solutions $u^\epsilon$ to the problem \eqref{eq: wave eq, w/noisy boundary} with mollified non-linearity $h^\epsilon$ - constructed in passing by the 2D non-linear Young equation - will converge uniformly as $\epsilon\to 0$, independent of the sequence of mollifications chosen. It will be in this sense that we solve the problem \eqref{eq: wave eq, w/noisy boundary} for distributional non-linearities $h$. \\
	\\
	After this brief motivation,  let us proceed to introduce the aforementioned transformations. Assume for some smooth $h^\epsilon$, we have a solution $u^\epsilon$ to 
	\begin{equation}\label{wave regularized}
		\begin{split}
			\left( \frac{\partial^2}{\partial x^2}- \frac{\partial^2}{\partial y^2}\right)u^\epsilon=h^\epsilon(u^\epsilon(x,y)),
		\end{split}
	\end{equation}
	that satisfies the Goursat boundary conditions \eqref{boundary wave eqn}. 
	Consider the transformation 
	\begin{equation}
		\begin{split}
			t_1&=\frac{y+x}{\sqrt{2}},\\
			t_2&=\frac{y-x}{\sqrt{2}},
			\label{transform}
		\end{split}
	\end{equation}
	which corresponds to the rotation operator $R_{-\frac{\pi}{4}} $
	and consider the function 
	\[
	\phi^\epsilon(t_1,t_2):=u^\epsilon\left( \frac{t_1-t_2}{\sqrt{2}},\frac{t_1+t_2}{\sqrt{2}}\right).
	\]
	Then it is easily verified that $\phi^\epsilon$ solves
	\begin{equation*}
		-2\frac{\partial^2}{\partial t_1 \partial t_2} \phi^\epsilon(t_1, t_2)=h^\epsilon\left(\phi^\epsilon(t_1, t_2)\right).
	\end{equation*}
	Moreover, note that if $t_2=0$, then $y-x=0$, which implies that 
	\[
	\phi^\epsilon(t_1, 0)=u^\epsilon\left(\frac{t_1}{\sqrt{2}}, \frac{t_1}{\sqrt{2}}\right)=\beta^1\left(\frac{t_1}{\sqrt{2}}\right).
	\]
	Similarly,  if $t_1=0$, then $y=-x$, i.e. 
	\[
	\phi^\epsilon(0,t_2)= u^\epsilon\left(\frac{-t_2}{\sqrt{2}}, \frac{t_2}{\sqrt{2}}\right)=\beta^2\left(\frac{t_2}{\sqrt{2}}\right).
	\]
	Hence, we derived from the wave equation with the boundary condition along characteristics the new boundary problem
	\begin{equation}\label{eq:Goursat type noisy boundary}
		-2\frac{\partial^2}{\partial t_1 \partial t_2} \phi^\epsilon=h^\epsilon\left(\phi^\epsilon(t_1,t_2)\right),
	\end{equation}
	with boundary condition 
	\begin{equation*}
		\begin{split}
			\phi^\epsilon(t_1, 0)&=\beta^1\left(\frac{t_1}{\sqrt{2}}\right),\\
			\phi^\epsilon(0, t_2)&=\beta^2\left(\frac{t_2}{\sqrt{2}}\right).
		\end{split}
	\end{equation*}
	Note that conversely, by employing the inverse transform to \eqref{transform}, solutions to \eqref{eq:Goursat type noisy boundary} give rise to solutions to \eqref{wave regularized}: If $\phi^\epsilon$ solves \eqref{eq:Goursat type noisy boundary}, then $u^\epsilon(x,y):=\phi^\epsilon((y-x)/\sqrt{2}, (y+x)/\sqrt{2})$ solves \eqref{wave regularized}. 
	
	It follows from the above derivation that by integration on both sides of \eqref{eq:Goursat type noisy boundary} using $t=(t_1,t_2)\in [0,T]^2$ (for some well chosen $T$)  we have 
	\begin{equation*}
		\phi_t^\epsilon=\beta^1\left(\frac{t_1}{\sqrt{2}}\right)+\beta^2\left(\frac{t_2}{\sqrt{2}}\right)-2\int_0^t h(\phi_s^\epsilon)\dd s,
	\end{equation*}
	where $\phi_t^\epsilon=u^\epsilon\left(\frac{t_1-t_2}{\sqrt{2}},\frac{t_1+t_2}{\sqrt{2}}\right)$. 
	Consider then $\psi^\epsilon=\phi^\epsilon_{\cdot}-\beta^1\left(\frac{\cdot}{\sqrt{2}}\right)-\beta^2\left(\frac{\cdot}{\sqrt{2}}\right)$, which then  solves the equation 
	\begin{equation}
		\psi_t^\epsilon=-2\int_0^t h^\epsilon\left(\psi^\epsilon_s+\beta^1\left(\frac{s_1}{\sqrt{2}}\right)+\beta^2\left(\frac{s_2}{\sqrt{2}}\right)\right)\dd s.
		\label{problem to solve}
	\end{equation}
	The above problem \eqref{problem to solve} can now alternatively be solved with the 2D non-linear Young theory as brought forward above, provided $\beta^1, \beta^2$ are sufficiently regularizing. To go back to the wave equation, we employ the reverse transform to \eqref{transform} to obtain
	\begin{equation}
		u^\epsilon(x,y):=\psi^\epsilon\left(\frac{y+2}{\sqrt{2}}, \frac{y-x}{\sqrt{2}}\right)+\beta^1\left(\frac{y+x}{\sqrt{2}}\right)+\beta^2\left(\frac{y-x}{\sqrt{2}}\right),
		\label{solution formula}
	\end{equation}
	as the unique solution to \eqref{wave regularized}. Note that as indeed $\psi^\epsilon(t_1, 0)=\psi^\epsilon(0, t_2)=0$ for any $t_1, t_2\in [0,T]$, we have that the boundary conditions \eqref{boundary wave eqn} are satisfied. Finally, if $h^\epsilon$ is issued from the mollification of some distribution $h$, i.e. $h^\epsilon=h*\rho^\epsilon$, we know by Proposition \ref{stability of solutions} that $(\psi^\epsilon)_\epsilon$ and thus $(u^\epsilon)_\epsilon$ will converge uniformly. This observation will serve us to define the following notion of solutions to \eqref{eq: wave eq, w/noisy boundary} even for distributional nonlinearities $h$.
	\begin{defn}
		Let $h\in \mathcal{S}(\mathbb{R})'$. We say that  $u$ is a solution to \eqref{eq: wave eq, w/noisy boundary} if for any sequence $(\rho^\epsilon)_\epsilon$ of mollifications, the sequence $(u^\epsilon)_\epsilon$ of solutions to \eqref{wave regularized} with mollified non-linearity $h^\epsilon=h*\rho^\epsilon$ converges to $u$ uniformly on $R_{\frac{\pi}{4}}\circ [0,T]^2$, i.e. in $C(R_{\frac{\pi}{4}}\circ [0,T]^2; \RR)$. 
		\label{notion of solution wave}
	\end{defn}
	Remark that with the above notion of solution, existence implies also uniqueness. Let us now pass to the main result of this section. 
	\begin{thm}\label{thm-wave equation}
		Let $\bar{\beta}^i(\cdot)=\beta^i(\cdot/\sqrt{2})$ and $p,q\in [1,\infty]$ such that $\frac{1}{p}+\frac{1}{q}=1$. Suppose $h\in B^\zeta_{p,p}(\RR)$ and suppose that $L_t(x)=(L^{-\bar{\beta}^1}_{t_1}*L^{-\bar{\beta}^2}_{t_2})(x)$ satisfies $L\in C^\gamma_t B^\alpha_{q,q}(\RR)$. Assume there exists $\eta\in (0,1)$ such that \[
		\gamma(1+\eta)>1 \qquad and \quad \zeta+\alpha>2+\eta .
		\]
		Then there exists a unique solution $u$ to \eqref{eq: wave eq, w/noisy boundary} in the sense of Definition \ref{notion of solution wave} given by 
		\[
		u(x,y):=\psi\left(\frac{y+2}{\sqrt{2}}, \frac{y-x}{\sqrt{2}}\right)+\beta^1\left(\frac{y+x}{\sqrt{2}}\right)+\beta^2\left(\frac{y-x}{\sqrt{2}}\right),
		\]
		where $\psi$ is the unique solution to 
		\begin{equation}
			\label{limit ode}
			\psi_t=-2\int_0^t (h*L)(ds, \psi_s),
		\end{equation}
		understood in the sense of Definition \ref{def: concept of solution}.
	\end{thm}
	\begin{proof}
		The above is an immediate consequence of Theorem \ref{thm:existence and uniqueness} as well as Proposition \ref{stability of solutions} in conjunction with the above considerations. Indeed, note that for any mollification $h^\epsilon=h*\rho^\epsilon$, the solution $u^\epsilon$ to \eqref{wave regularized} is given by \eqref{solution formula}. Note that by Corollary \ref{cor. stability}, $\psi^\epsilon$ given by \eqref{problem to solve} converges uniformly to $\psi$, the solution to 
		\eqref{limit ode}. Existence and uniqueness of $\psi$ is ensured by Theorem \ref{thm:existence and uniqueness}. Finally, by \eqref{solution formula}, this implies that $(u^\epsilon)_\epsilon$ converges uniformly to 
		\begin{equation*}
			u(x,y):=\psi\left(\frac{y+2}{\sqrt{2}}, \frac{y-x}{\sqrt{2}}\right)+\beta^1\left(\frac{y+x}{\sqrt{2}}\right)+\beta^2\left(\frac{y-x}{\sqrt{2}}\right),
		\end{equation*}
		completing the proof.
	\end{proof}
	Finally, we illustrate the above theorem  in the setting where the boundary processes $\beta^i$ are given as fractional Brownian motions in the following corollary. 
	\begin{cor}
		Let $h\in B^\zeta_{1,1}(\RR)$ for any $\zeta \in \RR$. Then there exist two independent fractional Brownian motions $\bar{\beta}^1, \bar{\beta}^2$ of  Hurst parameters, respectively, $H_1, H_2\in (0,1)$ such that problem \eqref{eq: wave eq, w/noisy boundary}  has a unique solution in the sense of Definition \ref{notion of solution wave}.
	\end{cor}
	\begin{proof}
		This is a direct consequence of Theorem \ref{thm:fbm reg} in combination with Theorem \ref{thm-wave equation}. 
	\end{proof}

	\begin{rem}\label{rem-referee}
		It is intriguing to note that one can partially connect Lemma \ref{lem:fbm reg-1} to the solution theory of non-linear wave equations with random initial data as developed by Burq, and Tzvetkov in \cite{BurqTzvetkov14}, see also the work of Bourgain \cite{Bourgain96} which is the first step in this direction. One of the key step in these works is to define a probability measure suitable Sobolev spaces $\cH^s$ corresponding to each initial data belonging to $\cH^s$. To understand the relation loosely let us set one variable as time, say $t_1$, and other as space, say $t_2$, then by taking $t_1=0$, $(\beta^2,0)$, where $\beta^2$ is a fBm with Hurst parameter $H \in (0,1)$, can be regarded as pair of initial data. 
		Since the law of $\beta^2$ defines a probability measure on $C([0,T]; \RR)$, the results of the current section (together with Lemma \ref{lem:fbm reg-1}) gives the existence of a unique solution, in suitable sense, to a class of wave equation, which gives \eqref{eq: wave eq, w/noisy boundary} under $R_{\frac{\pi}{4}}$,  with $(\beta^2,0)$ as random initial data. In particular, since $H \in (0,1)$ is arbitrary, we construct a set of probability measures $\cM$ on $C([0,T]; \RR)$, of size uncountable infinitely, such that for each measure $\mu \in \cM$,  a class of wave equation is locally well-posed with $\mu$ as random initial data taking values in  $C([0,T]; \RR)$. This connection is striking and pointed out to us by the referee. Since it is really fascinating and challenging to see if one can make this formal argument rigorous, we leave this line of research for future work. 
	\end{rem}

	\section{Further challenges, open problems and concluding remarks}\label{sec:challenges}
	We have extended the pathwise regularization by noise framework introduced in \cite{Catellier2016}  to equations on the plane, driven by a continuous regularizing field. The concept of regularization in this article has been presented in view of the local time associated with the field $w$. While we present here the case of regularization when the field $w$ is given by a fractional Brownian sheet or as a sum of two independent fractional Brownian motions, further systematic investigations of the space-time regularity of the local time associated to various stochastic fields appears in order.  In particular refined estimates for the local time of the fractional Brownian sheet using a  ``multiparameter Stochastic Sewing Lemma'' or the application of a stochastic reconstruction theorem as recently provided in \cite{kern} appear as an interesting direction for further research.
	
	Consider now a general stochastic partial differential equation of the form 
	\begin{equation*}
		Lu(t,x) =b(u(t,x))+ \dot{w}(t,x),\quad (t,x)\in [0,T]\times R, 
	\end{equation*}
	where $b$ is a non-linear function, $L$ is a differential operator,  and $\dot{w}$ is the formal mixed partial derivative a stochastic field $w$. We assume $R$ is a hyper-cube in $\RR^k$.  Given that $L$ generates a semi-group $\{S_t\}$,   mild solution can typically be written as a multi-parameter Volterra equation, in the sense that 
	\begin{equation*}
		u(t,x)=\xi(t,x)+\int_0^t\int_R S_{t-s}(x-y)b(u(s,y))\dd y\dd s+\int_0^t\int_R S_{t-s}(x-y)\dot{w}(s,y)\dd y \dd s. 
	\end{equation*}
	The equation above could be reformulated by similar principles as in the current article, although certain extensions must be made with respect to the construction of a non-linear Young integral to account for the possibly singular nature of the Volterra operator $S$. 
	The stochastic process obtained in $\int_0^t\int_R S_{t-s}(x-y)\dot{w}(s,y)\dd y \dd s$ might indeed provide a regularizing effect in this equation, but it is then also needed to investigate the space-time regularity of the local time associated with this field. The authors of \cite{Athreya20} have recently made certain progress in this direction the case where $L$ is the heat operator. There they prove regularization by noise when $\dot{w}$ is a white noise, and allow for distributional coefficients, including the Dirac delta.
	For more general differential operators the approach outlined above might yield interesting  results related to regularization by noise effects for a great variety of SPDEs.

	An alternative approach to that presented in the current article would be to study the regularity of the averaged field instead of only the local time. That is, one can study the regularity of the mapping 
	\begin{equation*}
		[0,T]^2\times \RR^d \ni(t,x)\mapsto \int_0^t b(x+w_r)\dd r,
	\end{equation*}
	for a given distribution $b$, and then use this instead of the convolution between $b$ and the local time $L^w$ as done in Theorem \ref{thm:existence and uniqueness}. As observed in \cite{Catellier2016} and further developed in \cite{galeati2020noiseless}, studying the averaged field directly in the concept of regularization by noise allows for less regularity requirements on the distribution $b$. In fact, in the one dimensional case when considering the SDE 
	\begin{equation*}
		x_t = x_0+\int_0^t b(x_s)\dd s + \beta_t,\quad t\in [0,T], 
	\end{equation*}
	where $\beta_t$ is a fractional Brownian motion with $H<\frac{1}{2}$ it is shown in \cite{Catellier2016,galeati2020noiseless} that pathwise existence and uniqueness as well as differentiability of the flow holds if $b\in C^\alpha_x$ (with compact support) and
	\begin{equation*}
		\alpha>2-\frac{1}{2H}. 
	\end{equation*}
	More recently in \cite{GaleatiGerencser22}, Galeati and Gerensc\'er push these results further to prove pathwise existence, uniqueness, and differentiability of the flow under the condition 
	\begin{equation*}
		\alpha >1-\frac{1}{2H}, 
	\end{equation*}
	without the assumption of compact support. 
	In contrast, using the local time approach presented in the current paper, in the one dimensional setting, \cite{HarangPerkowski2020} shows that a similar statement of existence and uniqueness holds if 
	\begin{equation*}
		\alpha>2+\frac{d}{2}-\frac{1}{2H}. 
	\end{equation*}
	The additional dimension dependency comes from the fact that the final set of $\bar{\Omega}\subset \Omega$ of full measure that that admits unique solutions is independent of the drift coefficient $b$, something which is an advantage in certain regularization by noise problems, see e.g.  \cite{bechtold} and \cite{harang2020pathwise}. \\

	\textbf{Acknowledgments:} The authors are grateful to Lucio Galeati for several fruitful discussions leading to an improvement of the results presented in this article. We wish to also thank an anonymous referee for several remark, suggestions and corrections that also greatly helped improve the results presented in this article. The first author acknowledges the funding from the European Research Council (ERC) under the European Union’s Horizon 2020 research and innovation program (grant agreement No.754362 and No.949981).The financial support by the German Science Foundation DFG through the Research Unit FOR 2402 is greatly acknowledged by the last author. \\
	
	\textbf{Declarations: }
	All the authors declare that they have no conflicts of interest. Data availability statement is not applicable in the context to the present article since all the results here are theoretical in nature and does not involve any data.


\end{document}